\documentclass{article}

\usepackage[utf8]{inputenc}
\usepackage[T1]{fontenc}

\usepackage{amsmath,amsfonts,amssymb,amsthm,cases,graphicx,authblk}
\usepackage{tikz}
\usepackage{tikz-3dplot}
\usepackage{pgfplots}
\usepackage{hyperref}
\usepackage{color}
\usepackage{vmargin}
\usepackage{nccmath}
\usepackage{verbatim}
\usepackage{float}
\usepackage{url}
\usepackage{cleveref}
\usepackage{fancyhdr}
\setmarginsrb{2cm}{1cm}{2cm}{3cm}{1cm}{1cm}{2cm}{2cm}
\usepackage[pagewise]{lineno}\linenumbers

\newcommand{\Id}{\operatorname{Id}}

\newcommand{\Cof}{\operatorname{Cof}}

\renewcommand{\leq}{\leqslant}

\renewcommand{\geq}{\geqslant}

\renewcommand{\Re}{\operatorname{Re}}

\everymath{\displaystyle}
\numberwithin{equation}{section}

\newtheorem{Theorem}{Theorem}[section]

\newtheorem{Corollary}[Theorem]{Corollary}
\newtheorem{Proposition}[Theorem]{Proposition}
\newtheorem{Lemma}[Theorem]{Lemma}
\newtheorem{Remark}[Theorem]{Remark}

\title{Local boundary feedback stabilization of a fluid-structure interaction problem under Navier slip boundary conditions with time delay} 

\author[1]{Imene Aicha Djebour \thanks{ E-mail address: \href{mailto: imene.djebour@univ-lorraine.fr}{imene.djebour@univ-lorraine.fr} }}
\affil[1]{Universit\'e de Lorraine, CNRS, Inria, IECL, F-54000 Nancy, France}

\date{\today}
\pgfplotsset{width=7cm,compat=1.8}

\begin{document}
	\maketitle	
\begin{abstract}
We consider a fluid-structure interaction system composed by a three-dimensional viscous incompressible fluid 
and an elastic plate located on the upper part of the fluid boundary. 
The fluid motion is governed by the Navier-Stokes system whereas the structure displacement satisfies the damped plate equation. We consider here the Navier slip boundary conditions. 
The main result of this work is the feedback stabilization of the strong solutions of the corresponding system around a stationary state for any exponential decay rate by means of a time delayed control localized on the fixed fluid boundary. This work relies on the Fattorini-Hautus criterion. Then, the main tool in this work is to show the unique continuation property of the associate solution to the adjoint system. 
\end{abstract} 
	
\vspace{1cm}
	
	

\tableofcontents
	
\section{Introduction}
We suppose that the fluid flow occupies a 3D periodic domain and an elastic plate that is disposed on the upper part of the fluid  boundary. The force exerted by the fluid on the plate influences the deformation of the elastic structure and then the fluid domain depends on the plate displacement and hence on time. Consequently, the interaction between the deformable structure and the fluid is modeled by a strongly coupled non linear system set in a moving domain. Our aim is to stabilize the position and the velocity of the structure as well as the velocity of the fluid  around a stationary state using a time delayed control that acts on a local part of the boundary of the fluid domain provided that the initial data are close enough to the stationary state in some norm. First, we give some important notations in order to write the system.

Let $\omega$ be the rectangular torus
\begin{equation*}
\omega=(\mathbb{R}/L_1\mathbb{Z})\times (\mathbb{R}/L_2\mathbb{Z})\quad L_1>0,\;L_2>0.
\end{equation*}
For any function $\eta : \omega \to (-1,\infty)$, we define
\begin{align*}
\Omega(\eta)&=\left\{ (x_1,x_2,x_3)\in \omega\times\mathbb{R} \ ; \ 
0<x_3<1+\eta (x_1,x_2)\right\}, \\
\Gamma(\eta)&=\left\{ (x_1,x_2,x_3)\in \omega\times\mathbb{R} \ ; \ 
x_3=1+\eta (x_1,x_2)\right\},\\
\Gamma_0 &=\omega\times\{0\}.
\end{align*}
(see the figure \ref{fig}).
In particular, we have
\begin{equation*}
\partial \Omega(\eta)=\Gamma(\eta)\cup \Gamma_0. 
\end{equation*}
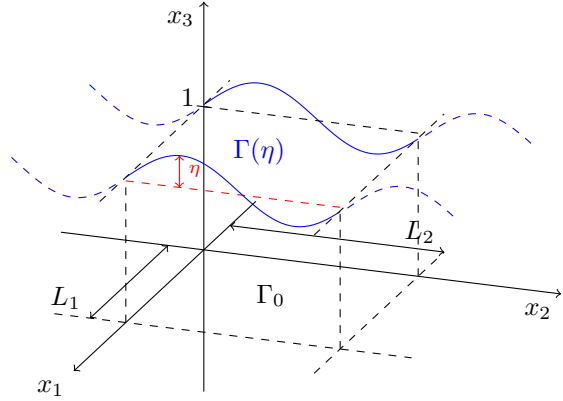
\begin{figure}[H]
	\tdplotsetmaincoords{70}{110}
	\centering
	\begin{tikzpicture}[tdplot_main_coords]
\draw[thin,->](-2,0,0)--(5,0,0)node[below left]{$x_1$};
\draw[thin,->](0,-2,0)--(0,5,0)node[below left]{$x_2$};
\draw[thin,->](0,0,-2)--(0,0,3.5)node[below left]{$x_3$};
\draw[dashed,-](3,-1,0)--(3,4,0);
\draw[dashed,-](4,3,0)--(-1,3,0);
\draw[dashed,-](3,3,0)--(3,3,2);
\draw[dashed,-](3,0,0)--(3,0,2);
\draw[scale=0.5,domain=0:6, smooth,variable=\x,blue] plot (0,\x,{0.8*sin(\x r)+4.1});
\begin{scope}[shift={(0,-3.1,0)}]
\draw[dashed,scale=0.5,domain=3:6, smooth,variable=\x,blue] plot (0,\x,{0.8*sin(\x r)+4.15});
\end{scope}
\begin{scope}[shift={(0,3.1,0)}]
\draw[dashed,scale=0.5,domain=0:3, smooth,variable=\x,blue] plot (0,\x,{0.8*sin(\x r)+4});
\end{scope}
\draw[dashed,-](0,3,0)--(0,3,2);
\draw[dashed,-](4,0,2.05)--(-1,0,2.05);
\draw[dashed,red,-](3,0,2)--(3,3,2);
\draw[thin,red,<->](3,0.75,2)--(3,0.75,2.45)node[below right]{\scriptsize{$\eta$}}; 
\draw[scale=0.5,domain=0:6, smooth,variable=\x,blue] plot (6,\x,{0.8*sin(\x r)+4.1});
\draw[dashed,-](4,3,1.95)--(-1,3,1.95);
\begin{scope}[shift={(0,-3.1,0)}]
\draw[dashed,scale=0.5,domain=3:6, smooth,variable=\x,blue] plot (6,\x,{0.8*sin(\x r)+4.15});
\end{scope}
\begin{scope}[shift={(0,3.1,0)}]
\draw[dashed,scale=0.5,domain=0:3, smooth,variable=\x,blue] plot (6,\x,{0.8*sin(\x r)+4});
\end{scope}
\draw[blue](2,2,2)node[above left]{$\Gamma(\eta)$};
\draw (2,2,0)node[above left]{$\Gamma_{0}$};
\draw[<->] (0,-0.5,0)--(3,-0.5,0)node[above left]{$L_1$};
\draw[<->] (-1,0,0)--(-1,3,0)node[above left]{$L_2$};
\draw (0,0,1.9)node[above left]{$1$};
\draw (0,-0.1,2.02)--(0,0.1,2.02);
\draw[dashed,-] (0,0,2.02)--(0,3,2.02);
\end{tikzpicture}
\caption{Configuration of the domain corresponding to the plate displacement $\eta$.}
\label{fig}
\end{figure}
We consider the following system describing the evolution of the fluid that is governed by the Navier-Stokes equations and the displacement of the elastic plate by the damped beam equation
\begin{equation}
\label{nav}
\left\{
\begin{array}{rl}
\partial_t U+(U\cdot \nabla)U-\nabla \cdot \mathbb{T}(U,P)=f^S   & \text{in}\ (0,\infty)\times \Omega(\eta),\\
\nabla \cdot U=0& \text{in} \ (0,\infty)\times \Omega(\eta),\\
\partial_{tt} \eta+\alpha\Delta^2\eta-\delta \Delta \partial_t \eta=\widetilde{\mathbb{H}}_\eta(U,P)+h^S & \text{in} \ (0,\infty)\times \omega.
\end{array}
\right.
\end{equation}
In \eqref{nav}, we have denoted by $U$ the fluid velocity, $P$ the fluid pressure and $\eta$ the transversal displacement of the elastic structure. The functions $f^S$ and $h^S$ are time independent data.

The Cauchy stress tensor $\mathbb{T}(U,P)$ is given by
$$ \mathbb{T}(U,P)=-P I_3 +2 \nu D(U),\quad D(U)_{i,j}=\frac{1}{2}\left( \frac{\partial  U_i}{\partial x_j}+\frac{\partial  U_j}{\partial x_i}\right).$$
The coefficients $\alpha$, $\delta$ and $\nu$ correspond respectively to the rigidity, the damping on the structure and the fluid viscosity.

We denote by $\widetilde{n}$ the unit exterior normal vector on $\partial \Omega(\eta)$ given as follows:
$$
\widetilde{n}=-e_3\;\text{on}\;\Gamma_{0},
$$
and on $\Gamma(\eta)$:
\begin{equation*}
\widetilde{n}(s,1+\eta(s))=\frac{\widetilde{N}(s,1+\eta(s))}{|\widetilde{N}(s,1+\eta(s))|}, \quad \text{where} \quad
\widetilde{N}(s,1+\eta(s))=
\begin{pmatrix}
-\partial_{s_1}\eta(s)\\-\partial_{s_2}\eta(s)  \\ 1
\end{pmatrix},
\quad s\in \omega.
\end{equation*} 
Here and in what follows, $|\cdot|$ stands for the Euclidean norm of $\mathbb{R}^k$, $k\geq 1$. Also, we select two tangent vectors $\widetilde{\tau}^i,\, i=1,2$ linearly independent on $\partial\Omega(\eta)$ such that
\begin{equation}
\label{tautild1}
\widetilde{\tau}^i=e_i,\quad i=1,2\text{  on }\Gamma_0,
\end{equation}
and 
\begin{equation}
\label{tautild2}
\widetilde{\tau}^1(s,1+\eta(s))=\begin{pmatrix}
1\\ 0 \\ \partial_{s_1}\eta(s)
\end{pmatrix},\quad \widetilde{\tau}^2(s,1+\eta(s))=\begin{pmatrix}
0\\ 1 \\ \partial_{s_2}\eta(s)
\end{pmatrix} \quad \text{ on $\Gamma(\eta)$}.
\end{equation}
The function $\widetilde{\mathbb{H}}_{\eta}$ is the contact force exerted by the fluid on the interface which is defined by
$$
\widetilde{\mathbb{H}}_\eta(U,P)=-\sqrt{1+|\nabla \eta|^2}\left( \mathbb{T}(U,P) \widetilde{n}\cdot e_3\right).
$$

We complete \eqref{nav} by the Navier slip boundary conditions. 
Let $a\in \mathbb{R}^3$, we denote by $a_{\widetilde{n}}$ and $a_{\widetilde{\tau}}$ the normal and the tangential components of $a$ :
\begin{equation}
\label{tak0.14}
a_{\widetilde{n}}= (a\cdot \widetilde{n})\widetilde{n}, \quad a_{\widetilde{\tau}} =a-a_{\widetilde{n}}=-\widetilde{n} \times \left( \widetilde{n} \times a\right).
\end{equation}

The Navier slip boundary conditions write as follows:
\begin{equation}
\label{navb}
\left\{
\begin{array}{rl}
U_{\widetilde{n}} = (\mathbb{M}v)_{\widetilde{n}}  & \text{on} \ (0,\infty)\times\Gamma_0,\\
\left[ 2\nu D(U) {\widetilde{n}}+\beta_{1} U\right]_{{\widetilde{\tau}}}=\beta_1(\mathbb{M}v)_{\widetilde{\tau}}& \text{on} \ (0,\infty)\times\Gamma_0,\\
(U-\partial_t \eta e_3)_{\widetilde{n}}=0&  \text{on} \ (0,\infty)\times\Gamma(\eta),\\
\left[ 2\nu D(U) {\widetilde{n}} + \beta_{2} \left(U-\partial_t \eta e_3\right)\right]_{{\widetilde{\tau}}}=0 &\text{on} \ (0,\infty)\times\Gamma(\eta).
\end{array}
\right.
\end{equation}
Here, $v$ is the control of the system \eqref{nav}, \eqref{navb} acting on the fixed boundary $\Gamma_0$. In order to preserve the compatibility condition due to the incompressibility of the fluid, we use the operator $\mathbb{M}$ defined by 
\begin{equation}
\label{defm}
\mathbb{M}v=mv-\left(\int_{\Gamma_0}m v\cdot {\widetilde{n}}\ d\Gamma\right)m{\widetilde{n}},
\end{equation}
where $m\in C^2(\Gamma_0)$ is compactly supported in $\Gamma_0$ which satisfies $\int_{\Gamma_0} m \ ds=1$, see for example \cite{MB1}, \cite{MR2247716}. Thus, the operator $\mathbb{M}$ localizes the action of the control in a relatively compact subset of $\Gamma_0$. We notice that $\mathbb{M}\in\mathcal{L}([L^2(\Gamma_0)]^3)$ and 
\begin{equation}
\label{com}
\int_{\Gamma_0}\mathbb{M}v \cdot\widetilde{n} \ d\Gamma=0.
\end{equation}

We assume that the friction coefficients $\beta_1$ and $\beta_2$ are two non negative constants 
$$
\beta_1\geq 0,\quad\beta_2\geq 0.
$$ 
Since $\omega$ is a rectangular torus, we complement the system  \eqref{nav}, \eqref{navb} with periodic boundary conditions for the fluid and for the elastic plate on the remaining boundaries of $\Omega(\eta)$. Then, we consider data and solutions which are periodic in the both directions $e_1$ and $e_2$, for example :
\begin{gather*}
U(t,x_1+L_1,x_2,x_3)=U(t,x_1,x_2,x_3), \quad U(t,x_1,x_2+L_2,x_3)=U(t,x_1,x_2,x_3), 
\\
\eta(t,s_1+L_1,s_2)=\eta(t,s_1,s_2), \quad \eta(t,s_1,s_2+L_2)=\eta(t,s_1,s_2).
\end{gather*}
Since $U$ is divergence free and taking into account the equation \eqref{com}, the velocity $\partial_t\eta$ should satisfy the condition
\begin{equation}
\label{com1}
\int_\omega\partial_t\eta \ ds =0.
\end{equation} 
Integrating the plate equation on $\omega$ and using \eqref{com1} and the fact that $\eta$ is periodic, we find
\begin{equation}
\label{com2}
\int_\omega \widetilde{\mathbb{H}}_\eta(U,P)+h^S \ ds=0.
\end{equation}
The above condition can be satisfied if we write 
\begin{equation}
\label{decom}
P=P_0+c,
\end{equation}
where $P_0$ satisfies
$$
\int_{\Omega(\eta)} P_0 \ dx =0,
$$
and $c$ is a constant that is chosen conveniently in such a way \eqref{com2} is verified.  
Consequently, the pressure $P$ will be uniquely defined.
In order to impose the condition \eqref{com2} without considering the normalizing constant $c$, we define the projection operator $M$ on $L^2_0(\omega)$ where
\begin{equation}
\label{l2}
L^2_0(\omega)=\left\lbrace \eta \in L^2(\omega)\ ; \ \int_{\omega}\eta \ ds=0\right\rbrace.
\end{equation}
Since we look for solutions such that \eqref{com1} is verified, it is convenient to consider the restriction of the associated semigroup to the plate equation to $L^2_0(\omega)$.
Then, in what follows, we assume that $\eta\in L^2_0(\omega)$.
To impose the condition \eqref{com2}, we substitute the plate equation by its projection on $L^2_0(\omega)$ using the operator $M$
$$
\partial_{tt} \eta+\alpha\Delta^2\eta-\delta \Delta \partial_t \eta=\mathbb{H}_\eta(U,P)+Mh^S,
$$
where
$$
\mathbb{H}_\eta(U,P)=M\widetilde{\mathbb{H}}_\eta(U,P).
$$
We complement the system \eqref{nav}, \eqref{navb} with the following initial conditions
\begin{equation}
\label{navc}
\left\{
\begin{array}{rl}
U(0,\cdot)=U^0 & \text{in} \  \Omega(\eta^0),\\
\eta(0,\cdot)=\eta^0 & \text{in}\ \omega,\\
\partial_t \eta(0,\cdot)=\eta^1& \text{in}\ \omega.\\
\end{array}
\right.
\end{equation}

Let $(w^S,p^S,\eta^S)$ be a stationary state of the system \eqref{nav}, \eqref{navb} that is a solution of the system
\begin{equation}
\label{stabs}
\left\{
\begin{array}{rl}
(w^S\cdot \nabla)w^S-\nabla \cdot \mathbb{T}(w^S,p^S)=f^S   & \text{in}\ \Omega(\eta^S),\\
\nabla \cdot w^S=0& \text{in} \  \Omega(\eta^S),\\
\alpha\Delta^2\eta^S=\mathbb{H}_{\eta^S}(w^S,p^S)+Mh^S & \text{in} \  \omega,\\
w^S_{n} = 0  & \text{on} \ \partial\Omega(\eta^S),\\
\left[ 2\nu D(w^S) {n}+\beta w^S\right]_{\tau}=0& \text{on} \ \partial\Omega(\eta^S),
\end{array}
\right.
\end{equation}
where 
$$
\beta(y)=
\left\{
\begin{array}{ccccc}
\beta_1 &\text{if}& y\in \Gamma_0,\\
\beta_2 &\text{if}& y\in \Gamma(\eta^S).
\end{array}
\right.
$$
The vector $n$ stands for the unitary exterior normal vector  and $\tau^i$, $i=1,2$ designate two tangent vectors on $\partial\Omega(\eta^S)$:
\begin{equation}
\label{tau1*}
n=-e_3,\quad\tau^i=e_i,\quad i=1,2\text{  on }\Gamma_0,
\end{equation}
and on $\Gamma(\eta^S)$:
\begin{equation}
\label{tau3*}
n(s,1+\eta^S(s))=\frac{N(s,1+\eta^S(s))}{|N(s,1+\eta^S(s))|} \quad \text{where} \quad
N(s,1+\eta^S(s))=
\begin{pmatrix}
-\partial_{s_1}\eta^S(s)\\-\partial_{s_2}\eta^S(s)  \\ 1
\end{pmatrix}
\quad s\in \omega,
\end{equation} 
\begin{equation}
\label{tau2*}
\tau^1(s,1+\eta^S(s))=\begin{pmatrix}
1\\ 0 \\ \partial_{s_1}\eta^S(s)
\end{pmatrix},\quad \tau^2(s,1+\eta^S(s))=\begin{pmatrix}
0\\ 1 \\ \partial_{s_2}\eta^S(s)
\end{pmatrix}.
\end{equation}

Our aim is to stabilize the solution $(U,P,\eta)$ of \eqref{nav}, \eqref{navb}, \eqref{navc} around the stationary state $(w^S,p^S,\eta)$ that is a solution of \eqref{stabs} by means of a feedback boundary control $v(t)$ that appears in \eqref{navb} depending on the state $(U,\eta)$. In practice, due to a calculus time issues, we include a time delay $t_0>0$ in the control and the main goal is to construct this feedback control in such a way it depends at time $t$ on the values of the state $(U(t'),\eta(t'))$ for $t'\leq t-t_0$  and the associated strong solution $(U,P,\eta)$  to the system \eqref{nav}, \eqref{navb}, \eqref{navc} goes to $(w^S,p^S,\eta^S)$ exponentially.
This context of stabilization has an important interest in the control theory of PDE: The analysis of the effect of the time delay on the feedback stabilization of some partial differential equations is given for instance in \cite{MR818942} and \cite{MR937679}. Recently, in \cite{MR3767485} a feedback control is constructed for a finite dimensional system with input delay and in \cite{MR3936420} the authors manage to obtain a stabilizing feedback boundary control to the one-dimensional reaction-diffusion system considering a constant delay, their method relies essentially on the Arstein transform and the fact that the generator operator of the system is self-adjoint. We mention also the work \cite{lhachemi2019boundary} where the author considered the stabilization of a damped Euler-Bernoulli beam equation with a time delayed control, here the system is parabolic but it is not generated by a self-adjoint operator. In \cite{hal-02545562}, the authors extend the theory developed in \cite{MR3936420} and construct a finite dimensional feedback stabilizing control with input delay to a general class of parabolic systems. The challenge in the present work is to apply the theory of \cite{hal-02545562} to deduce the stabilization of the fluid-structure problem \eqref{nav}, \eqref{navb} since it describes a parabolic system thanks to the damping term $-\delta \Delta \partial_t \eta$ that appears in the structure equation in \eqref{nav}. 

In order to state the main result, we give some notations.
As it is standard in the studies of fluid-structure interaction systems, one of the main difficulties lies in the fact that the spatial domain of the fluid is variable and unknown. Since the problem consists to compare the asymptotic behavior of the solution $(U,P,\eta)$ to the stationary state $(w^S,p^S,\eta^S)$, we are led to introduce a change of variable to transform the functions $(U,P)$ into $(\widetilde{u},\widetilde{p})$ defined in the fixed domain $\Omega(\eta^S)$. More precisely, we set
\begin{equation}
\label{chg21}
X_{\eta^1,\eta^2} : \Omega(\eta^1) \longrightarrow \Omega(\eta^2), \quad 
\begin{pmatrix} y_1 \\ y_2\\ y_3 \end{pmatrix} \longmapsto \begin{pmatrix}  y_1 \\ y_2\\ \dfrac{1+\eta^2(y_1,y_2)}{1+\eta^1(y_1,y_2)} y_3 \end{pmatrix}.
\end{equation}
We consider the change of variables
\begin{equation}\label{chg2}
X(t,\cdot)=X_{\eta^S,\eta(t,\cdot)}, \quad Y(t,\cdot)=X_{\eta(t,\cdot),\eta^S},
\end{equation}
and we set
$$
\Omega=\Omega(\eta^S).
$$
We define the new functions

\begin{equation}
\label{chg3}
\widetilde{u}(t,y)=(\Cof\nabla X(t,y))^*U(t,X(t,y)) , \quad \widetilde{p}(t,y)=P(t,X(t,y)) \quad (t\geq 0, \ y\in \Omega).
\end{equation}
We observe that we have
$$
\widetilde{u}^0=(\Cof\nabla X(0,\cdot))^*U^0(X(0,\cdot)). 
$$

In what follows, we recall that $n$ is the unitary exterior normal vector  and $\tau^i$, $i=1,2$ are two tangent vectors on $\partial\Omega$ already defined in \eqref{tau1*}, \eqref{tau2*} and \eqref{tau3*}.

The operators $(A_1,\mathcal{D}(A_1))$ and $(A_2,\mathcal{D}(A_2))$  are defined by 
\begin{align}
\label{A1A2*}
A_1 \eta=  \alpha\Delta^2\eta,
\quad \mathcal{D}(A_1)=H^4(\omega) \cap L^2_0(\omega),\\
A_2 \eta=-\delta \Delta \eta,
\quad \mathcal{D}(A_2)=H^2(\omega) \cap L^2_0(\omega).
\end{align}
Let $\mathfrak{X}_1$, $\mathfrak{X}_2$ be two Banach spaces with the norms $\left\| .\right\|_{\mathfrak{X}_1}$ and $\left\| .\right\|_{\mathfrak{X}_2}$ respectively. For $s\geq 0$, we define
$$W^{s}(0,\infty;\mathfrak{X}_1,\mathfrak{X}_2)=\left\lbrace v\in L^2(0,\infty;\mathfrak{X}_1)\ ; \ v\in H^s(0,\infty;\mathfrak{X}_2)\right\rbrace, $$
with
$$\left\|.\right\|_{W^{s}(0,\infty;\mathfrak{X}_1,\mathfrak{X}_2)}=\left\|.\right\|_{L^2(0,\infty;\mathfrak{X}_1)}+\left\|.\right\|_{H^s(0,\infty;\mathfrak{X}_2)}.$$
For $s=1$, the space $W^1(0,\infty;\mathfrak{X}_1,\mathfrak{X}_2)$ is denoted by $W(0,\infty;\mathfrak{X}_1,\mathfrak{X}_2)$. 

Let $f(t)$ be a vector valued function. For $\gamma>0$, we define $f_\gamma$ by
$$
f_\gamma: t\mapsto e^{{\gamma} t}f(t). 
$$	
Then, for $\gamma>0$, we consider the following spaces
\begin{equation}
\label{L2g4}
L^p_{{\gamma}}(0,\infty;\mathfrak{X}_1)=\{f\in L^p(0,\infty;\mathfrak{X}_1) \ ; \ f_\gamma\in L^p(0,\infty;\mathfrak{X}_1)\},\quad p\in[1,+\infty],
\end{equation}
and
\begin{equation}
\label{Wg4}
W_{{\gamma}}^{s}(0,\infty;\mathfrak{X}_1,\mathfrak{X}_2)=\{f\in W^{s}(0,\infty;\mathfrak{X}_1,\mathfrak{X}_2) \ ; \ f_{\gamma}\in W^{s}(0,\infty;\mathfrak{X}_1,\mathfrak{X}_2)\},
\end{equation}
with the norms
$$
\left\|f\right\|_ {L^p_{\gamma}(0,\infty;\mathfrak{X}_1)}=\left\|f_{\gamma}\right\|_{L^p(0,\infty;\mathfrak{X}_1)},
$$
$$
\left\|f\right\|_{W_{\gamma}^{s}(0,\infty;\mathfrak{X}_1,\mathfrak{X}_2)}=\left\| f_{\gamma}\right\|_ {W^{s}(0,\infty;\mathfrak{X}_1,\mathfrak{X}_2)}. 
$$
We set
\begin{equation}
\label{spceXg4}
\mathcal{X}_{\infty,\gamma}=W_{\gamma}(0,\infty;[H^2(\Omega)]^3,[L^2(\Omega)]^3)  \times L_{\gamma}^2(0,\infty;H^1(\Omega)/\mathbb{R})\times W^2_{\gamma}(0,\infty;\mathcal{D}(A_1),L_0^2(\omega)),
\end{equation}
such that
\begin{multline}\label{tak3.14}
\left\|(u,p,\eta)\right\|_{\mathcal{X}_{\infty,\gamma}}=\left\| u\right\|_{W_{\gamma}(0,\infty;[H^2(\Omega)]^3,[L^2(\Omega)]^3)}
+\left\| u\right\|_{L_{\gamma}^{\infty}(0,\infty;[H^1(\Omega)]^3)}
+\left\|\nabla p\right\|_{L_{\gamma}^2(0,\infty,[L^2(\Omega)]^3)}
\\
+\left\|\eta\right\|_{W_{\gamma}^2(0,\infty;\mathcal{D}(A_1),L_0^2(\omega))}
+\left\|\eta\right\|_{L^\infty_{\gamma}(0,\infty;H^3(\omega))}
+\left\|\partial_t \eta\right\|_{L_{\gamma}^\infty(0,\infty;H^1(\omega))}.
\end{multline}
We define the space
$$
\mathcal{H}=[L^2(\Omega)]^3\times \mathcal{D}(A_1^{1/2})\times L^2_0(\omega),
$$
furnished with the inner product
$$
\left\langle \begin{pmatrix}
\phi\\\zeta_1\\\zeta_2
\end{pmatrix},\begin{pmatrix}
\psi\\\xi_1\\\xi_2
\end{pmatrix} \right\rangle_{\mathcal{H}}= \int_\Omega \phi\cdot \psi\ dy+\int_\omega A_1^{1/2}\zeta_1\cdot A_1^{1/2}\xi_1 \ ds+\int_\omega \zeta_2\cdot \xi_2\ ds.
$$
We set 
\begin{equation}
\label{grw}
\widetilde{W}=\begin{pmatrix}
\widetilde{u}-w^S\\ \eta-\eta^S \\ \partial_t\eta
\end{pmatrix},
\end{equation}
where the function $\widetilde{u}$ is defined by \eqref{chg3}.

Finally, we suppose that the initial conditions verify
\begin{equation}
\label{co1}
(\widetilde{u}^0,\eta^0,\eta^1)\in [H^1(\Omega)]^3\times H^3(\omega)\times H^1(\omega)
\end{equation}
\begin{equation}
\label{co}
\nabla\cdot\widetilde{u}^0=0 \text{ in} \ \Omega,\quad (\widetilde{u}^0-\eta^1e_3)_{n}=0 \ \text{on} \ \Gamma(\eta^S),\quad \widetilde{u}^0_{n}=0 \ \text{on} \ \Gamma_0.
\end{equation}
Now, we state the main result of this paper.
\begin{Theorem}\label{thmain2}
Let $t_0>0$, $\gamma>0$,
\begin{equation}
f^S\in [W^{2,\infty}(\Omega)]^3,\quad h^S\in L^2(\omega),
\end{equation}
\begin{equation}
\label{regstat}
(w^S,p^S,\eta^S)\in [W^{2,\infty}(\Omega)]^3\times W^{1,\infty}(\Omega)/ \mathbb{R}\times C^4(\omega) ,
\end{equation} 
with
\begin{equation}
\label{smallnes}
1+\eta^S>0,
\end{equation}
such that the system \eqref{stabs} is verified and $(\widetilde{u}^0,\eta^0,\eta^1)$ satisfying \eqref{co1}, \eqref{co}.
Then, there exist
$$
N_\gamma\in\mathbb{N}^*,\quad K\in L^\infty_{\rm loc}(\mathbb{R}^2;\mathcal{L}(\mathcal{H})),
$$
$$
(\phi_k,\zeta^k_1,\zeta_2^k)\in [H^2(\Omega)]^3\times \mathcal{D}(A_1)\times \mathcal{D}(A_1^{1/2}),\ v_k\in [H^{3/2}(\Gamma_0)]^3,\;k=1,\ldots,N_\gamma,
$$ 
and $R>0$, 
such that if 
$$
\left\|\widetilde{u}^0-w^S \right\|_{[H^1(\Omega)]^3}+ \left\|\eta^0-\eta^S \right\|_{H^3(\omega)}+\left\|\eta^1\right\|_{H^1(\omega)}\leq R,
$$
there exists a unique strong solution $(U,P,\eta)$ of \eqref{nav}, \eqref{navb} and \eqref{navc} associated to
\begin{equation}
\label{fb3}
v(t)=1_{[t_0,+\infty)}(t) \sum_{k=1}^{N_\gamma} \left(\widetilde{W}(t-t_0)+\int_0^{t-t_0} K(t-t_0,s)\widetilde{W}(s) \ ds,\begin{pmatrix}
\phi_k\\\zeta^k_1\\\zeta_2^k
\end{pmatrix}\right)_{\mathcal{H}} v_k,
\end{equation}
where $\widetilde{W}$ is given by \eqref{grw}.

Moreover, $(\widetilde{u}-w^S,\widetilde{p}-p^S,\eta-\eta^S) \in \mathcal{X}_{\infty,\gamma}$ where the functions $\widetilde{u}$, $\widetilde{p}$ are defined by \eqref{chg3} and we have the following estimate
\begin{equation}\label{ijt202}
\left\|(\widetilde{u}-w^S,\widetilde{p}-p^S,\eta-\eta^S)\right\|_{\mathcal{X}_{\infty,\gamma}}
\leq C\left(\left\|\widetilde{u}^0-w^S \right\|_{[H^1(\Omega)]^3}+ \left\|\eta^0-\eta^S \right\|_{H^3(\omega)}+\left\|\eta^1\right\|_{H^1(\omega)}\right).
\end{equation}

\end{Theorem}

\begin{Remark}
We can obtain analogous results if we consider the classical Dirichlet boundary conditions 
\begin{equation}
\label{dir}
\left\{
\begin{array}{rl}
U=\mathbb{M}v & \text{on} \ (0,\infty)\times\Gamma_0,\\
(U-\partial_t \eta e_3)=0&  \text{on} \ (0,\infty)\times\Gamma(\eta).\\
\end{array}
\right.
\end{equation}	
We underline that if $\beta_{1}=0$, the \cref{thmain2} implies that the system \eqref{nav}, \eqref{navb} and \eqref{navc} is exponentially stabilizable by considering only a scalar control in the impermeability boundary condition . 
\end{Remark}
\begin{Remark}
Here, the regularity of $\eta^S$ ensures to have a sufficient regular domain to apply elliptic estimates of the Stokes problem with Navier-slip boundary conditions proved in \cite{BdV}. 
\end{Remark}
In this paper, we are interested in the feedback stabilization of the system \eqref{nav}, \eqref{navb} and \eqref{navc} that is similar to the system studied in \cite{MT} and \cite{JPR} except that here, we consider the Navier boundary conditions. These conditions were introduced in \cite{navier1823} by Navier in 1823. The consideration of this type of boundary conditions is significant in many physical aspects, see for instance  \cite{MR2123407, MR884296, doi:10.1002/fld.1650040302}. Note that the existence and uniqueness of the strong solution for the system \eqref{nav}, \eqref{navb}, \eqref{navc} have been obtained recently in \cite{MR3962841}. Concerning the stabilization without delay for fluid structure problems, there are few results in the literature. In \cite{JPR}, the author considered a coupling system that models the interaction between an incompressible  fluid governed by the Navier-Stokes system and an elastic structure that enjoys the damped Euler-Bernoulli equation in a rectangular domain considering the Dirichlet boundary conditions, the control in this case is distributed along the structure. The author shows the exponential stabilization of the strong solution of the considered system around zero. The same system was studied in \cite{MT}, where the authors considered a boundary control acting on a located part of the fluid boundary. The authors obtained a stabilization result of the weak solution around a stationary state. In the case where the structure is a rigid body, we have \cite{MR3181675}, \cite{MR3261920}.

In this work, we construct a boundary feedback control $v$ with a time delay $t_0>0$ (in the direction of \cite{hal-02545562}) such that the strong solution of our system is exponentially stabilizable around a fixed stationary state. Due to regularity concerns, compatibility conditions at $t=0$ between $v$ and the initial conditions should be imposed, see for example \cite{MR2567253,MR2516189,MR2851895}. To  overcome this difficulty, some strategies was developed in \cite{MR2335090,MB1}. In the case of a time delay, the control acts on the system from $t_0>0$. Consequently, the control vanishes at $t=0$. Then, the compatibility conditions are automatically satisfied and this is a good advantage of the time delay.

The paper is organized as follows: in \cref{se1}, we write the system \eqref{nav}, \eqref{navb}, \eqref{navc} in a fixed domain using the change of variables \eqref{chg2} and \eqref{chg3}.  Next, we study the associated linear system in \cref{sl} that can be written as  
\begin{equation}\label{ijt000*}
z'=Az+Bv+f, \quad z(0)=z^0.
\end{equation} 
Taking into account the \cref{r}, the linearized problem is written as a system coupling the Oseen equations and the damped plate equation, disturbed by some linear terms appearing in the whole system even in the boundary conditions. We show that the infinitesimal operator associated to the linear system is analytic and we establish a first result of a time delayed stabilization of the linearized problem. In fact, we show that the Fattorini-Hautus criterion is satisfied
\begin{equation}
\forall\varepsilon\in \mathcal{D}(A^*), \ 
\forall \lambda\in\mathbb{C},\ 
\Re\lambda\geq -\sigma\quad
A^*\varepsilon= \lambda \varepsilon \quad \text{et} \quad B^*\varepsilon=0
\quad\Longrightarrow\quad 
\varepsilon = 0.  \tag{$\text{UC}_{\sigma}$}\label{UCstab*}
\end{equation}
More precisely, the condition \eqref{UCstab*} is obtained using the unique continuation property of an auxiliary Oseen system with Neumann boundary condition. 
Finally, we prove the \cref{thmain2} applying the Banach fixed point argument.
\section{Change of variables}
\label{se1}
Let us define $\mathcal{T}_\eta$  
\begin{equation}
\label{teta1}
(\mathcal{T}_{\eta}\xi)(y)=
\left\{
\begin{array}{ccccc}
0 &\text{if}& y\in \Gamma_0,\\
\xi(s) e_3 &\text{if}& y=(s,1+\eta(s))\in \Gamma(\eta),
\end{array}
\right.
\end{equation}
for any function $\eta: \omega\longrightarrow(-1,\infty)$. We notice that  $\mathcal{T}_{\eta}\in \mathcal{L}(L^2(\omega);[L^2(\partial \Omega(\eta))]^3)$
and
\begin{equation*}
(\mathcal{T}^*_{\eta}\zeta)(s)=\sqrt{1+|\nabla \eta|^2}\zeta(s,1+\eta(s))\cdot e_3, \quad \forall \zeta\in[L^2(\partial \Omega(\eta))]^3.
\end{equation*}
We define $\mathcal{T}$ by 
$$
\mathcal{T}=\mathcal{T}_{\eta^S}M,
$$
where we recall that $M$ is the orthogonal projection on the space $L^2_0(\omega)$ defined by \eqref{l2}.

Using \eqref{chg2} and \eqref{chg3}, the system \eqref{nav}, \eqref{navb}, \eqref{navc} is equivalent to
\begin{equation}\label{nav2}
\left\{
\begin{array}{rl}
\partial_t \widetilde{u}-\nabla\cdot\mathbb{T}(\widetilde{u},\widetilde{p})=\det(\nabla X)f^S(X)+(\Id-\mathbf{K}_\eta)\partial_t \widetilde{u}
-\nu(\Delta-\mathbf{L}_\eta)\widetilde{u}\\+(\nabla-\mathbf{G}_\eta)\widetilde{p} +\mathbf{M}_\eta\widetilde{u}+\mathbf{N}_\eta\widetilde{u} & \text{in}\ (0,\infty)\times \Omega,\\
\nabla \cdot \widetilde{u}=0  & \text{in}\ (0,\infty)\times \Omega,\\
\partial_{tt} \eta+A_1\eta +A_2 \partial_t \eta=\mathbb{H}_{\eta^S}(\widetilde{u},\widetilde{p})+H(\widetilde{u},\eta)+Mh^S &\text{in}\ (0,\infty)\times \omega, \\
\end{array}
\right.
\end{equation}
with the boundary conditions
\begin{equation}\label{nav2b}
\left\{
\begin{array}{rl}
\left[\widetilde{u}-\mathcal{T}\partial_t \eta\right]_{n}=1_{\Gamma_0}(\mathbb{M}v)_{n} &  \text{on}\ (0,\infty)\times \partial\Omega,\\
\left[2\nu D(\widetilde{u})n+\beta (\widetilde{u}-\mathcal{T}\partial_t \eta)\right]_{\tau}=1_{\Gamma_0}\beta (\mathbb{M}v)_{\tau}+G(\widetilde{u},\eta) &  \text{on}\ (0,\infty)\times \partial\Omega,
\end{array}
\right. 
\end{equation}
 and with the initial conditions
\begin{equation}\label{nav2c}
\left\{
\begin{array}{rl}
\widetilde{u}(0,\cdot )=(\Cof\nabla X(0,.))^*U^0(X(0,\cdot))=\widetilde{u}^0  &\text{in}\ \Omega,\\
\eta(0, \cdot)=\eta^0 &\text{in}\  \omega,\\
\partial_t \eta(0,\cdot)=\eta^1&\text{in}\  \omega.\\
\end{array}
\right.
\end{equation}
We precise here that $1_{\Gamma_0}$ is the characteristic function of the set $\Gamma_{0}$.
To simplify the notations, we set
$$
a=(\Cof(\nabla Y))^*, \quad b=(\Cof(\nabla X)),
$$
$$
\mathbf{K}_\eta\widetilde{u}=(\nabla X) \widetilde{u},
$$
$$
\mathbf{G}_\eta\widetilde{p}=b\nabla \widetilde{p}.
$$
We have the following formulas
\begin{multline}
\label{defFS}
[-\nu(\Delta-\mathbf{L}_\eta)\widetilde{u}]_i
=
\nu\sum_{j,k,l,m}\left( \det(\nabla X)a_{ik}(X)\frac{\partial Y_m}{\partial x_j}(X)\frac{\partial Y_l}{\partial x_j}(X)-\delta_{ik}\delta_{mj}\delta_{jl}\right) 
\frac{\partial^2\widetilde{u}_k}{\partial y_l\partial y_m}
\\
+\nu\sum_{j,k,l,m}\left( \det(\nabla X)a_{jk}(X)\frac{\partial Y_m}{\partial x_i}(X)\frac{\partial Y_l}{\partial x_j}(X)-\delta_{jk}\delta_{mi}\delta_{jl}\right) 
\frac{\partial^2\widetilde{u}_k}{\partial y_l\partial y_m}
\\
+\det(\nabla X)\Bigg[\nu\sum_{j,k,l}\left(\frac{\partial a_{ik}}{\partial x_j}(X)\frac{\partial Y_l}{\partial x_j}(X)+\frac{\partial a_{jk}}{\partial x_i}(X)\frac{\partial Y_l}{\partial x_j}(X)\right) \frac{\partial \widetilde{u}_k}{\partial y_l}
\\
+\nu\sum_{j,k,l}\bigg(\frac{\partial a_{ik}}{\partial x_j}(X)\frac{\partial Y_l}{\partial x_j}(X)+a_{ik}(X)\frac{\partial ^2Y_l}{\partial x_j^2}(X)+\frac{\partial a_{jk}}{\partial x_j}(X)\frac{\partial Y_l}{\partial x_i}(X)\\+a_{jk}(X)\frac{\partial ^2Y_l}{\partial x_j\partial x_i}(X)\bigg) \frac{\partial \widetilde{u}_k}{\partial y_l}
+\nu\sum_{k}\left( \frac{\partial^2a_{ik}}{\partial x_j^2}(X)+\frac{\partial^2a_{jk}}{\partial x_j\partial x_i}(X)\right) \widetilde{u}_k\Bigg],
\end{multline}
\begin{equation}
[(\nabla-\mathbf{G}_\eta)\widetilde{p}]_i=\sum_k (\delta_{ik}-b_{ik})\frac{\partial \widetilde{p}}{\partial y_k}=\sum_k(\delta_{ik}-\det(\nabla X)\frac{\partial Y_k}{\partial x_i}(X))\frac{\partial \widetilde{p}}{\partial y_k},
\end{equation}
\begin{equation}
[\mathbf{N}_\eta\widetilde{u}]_i=-\sum_{k,l,j}\det(\nabla X)a_{kl}(X)\frac{\partial a_{ij}(X)}{\partial x_k}\widetilde{u}_l\widetilde{u}_j
-\sum_{k,l,j,m}\det(\nabla X)a_{kl}(X)a_{ij}(X)\frac{\partial Y_m}{\partial x_k}(X)\widetilde{u}_l\frac{\partial \widetilde{u}_j}{\partial y_m},
\end{equation}
\begin{equation}
[\mathbf{M}_\eta\widetilde{u}]_i=-\sum_{l,k}\det(\nabla X)a_{ik}(X)\frac{\partial \widetilde{u}_k}{\partial y_l}\partial_tY_l(X)
\\-\sum_k\det(\nabla X)\partial_t a_{ik}(X)\widetilde{u}_k,\quad i=1,2,3.
\end{equation} 
The non linear term $H$ appearing in the plate equation writes 
\begin{multline}
\label{defHS}
H(\widetilde{u},\eta)=\nu M\Bigg[
-\sum_{j,k}\left( \frac{\partial a_{3k}}{\partial x_j}(X)+\frac{\partial a_{jk}}{\partial x_3}(X)\right)\widetilde{N}_j \widetilde{u}_k
\\+\sum_{j,k,l} \left(\delta_{3k}\delta_{jl}(N)_j-a_{3k}(X)\frac{\partial Y_l}{\partial x_j}(X)\widetilde{N}_j\right)\frac{\partial \widetilde{u}_k}{\partial y_l}
\\
+\left( \delta_{3l}\delta_{jk}(N)_j-a_{jk}(X)\frac{\partial Y_l}{\partial x_3}(X)\widetilde{N}_j\right)\frac{\partial \widetilde{u}_k}{\partial y_l}\Bigg].
\end{multline}
Now let us deal with the boundary conditions. Let $\mathcal{W}$ be the operator defined by 
\begin{multline}
\label{tak4.1*}
[\mathcal{W}(\widetilde{u},\eta)]_k
= \nu\sum_{j,m} \widetilde{n}_j(X)\left( \frac{\partial a_{km}}{\partial x_j}(X)\widetilde{u}_m+\frac{\partial a_{jm}}{\partial x_k}(X)\widetilde{u}_m\right)
+\beta\left( \sum_j a_{kj}(X) \widetilde{u}_j-\mathcal{T}\partial_t\eta\cdot e_k\right)
\\
+\nu\sum_{j,m,q} \widetilde{n}_j(X)\left( a_{km}(X)\frac{\partial \widetilde{u}_m}{\partial y_q}\frac{\partial Y_q}{\partial x_j}(X)
+a_{jm}(X)\frac{\partial \widetilde{u}_m}{\partial y_q}\frac{\partial Y_q}{\partial x_k}(X)\right) ,
\quad k=1,2,3.
\end{multline} 
The boundary conditions \eqref{navb} become after the change of variables \eqref{chg2} and \eqref{chg3}
\begin{equation}
\label{sof}
\left\{
\begin{array}{rl}
(\widetilde{u}-\mathcal{T}\partial_t\eta)_{n} =1_{\Gamma_0} (\mathbb{M}v)_{n}  & \text{on} \ (0,\infty)\times\partial\Omega,\\
\mathcal{W}(\widetilde{u},\eta)\cdot(\widetilde{\tau}^i(X))=1_{\Gamma_{0}}\beta(\mathbb{M}v)\cdot(\widetilde{ \tau}^i(X)),\quad i=1,2 &\text{  on } (0,\infty)\times\partial\Omega,
\end{array}
\right.
\end{equation}
where $\widetilde{ \tau}^i$ are defined by \eqref{tautild1} and \eqref{tautild2}. The second boundary condition in \eqref{sof} can be written as
\begin{equation*}
\left( 2\nu D(\widetilde{u})n+\beta(\widetilde{u}-\mathcal{T}\partial_t \eta)\right)\cdot\tau^i=\mathcal{V}^i(\widetilde{u},\eta)+1_{\Gamma_{0}}\beta(\mathbb{M}v)\cdot\tau^i,\quad i=1,2,\ \text{ on }\partial\Omega,
\end{equation*}
where $\tau^i$ are the two tangent vectors on $\partial\Omega$ defined by \eqref{tau1*}, \eqref{tau2*} and $\mathcal{V}^i$ is given by
\begin{multline}
\label{tak4.2S}
\mathcal{V}^i(\widetilde{u},\eta)=
\left( 2\nu D(\widetilde{u})n+\beta(\widetilde{u}-\mathcal{T}\partial_t \eta)\right)\cdot(\tau^i-\widetilde{\tau}^i(X))\\+(2\nu D(\widetilde{u})n+\beta(\widetilde{u}-\mathcal{T}\partial_t \eta)-\mathcal{W}(\widetilde{u},\eta))\cdot \widetilde{\tau}^i(X),\quad i=1,2,\ \text{ on }\partial\Omega.
\end{multline}
We construct an operator $G$ such that $G\cdot n=0$ and $G\cdot\tau^i=\mathcal{V}^i$ on $\partial\Omega$. Consequently, $G$ is defined on $\Gamma(\eta^S)$ by
\begin{equation}
\label{defGS}
\begin{array}{rcl}
G_1&=&\frac{\mathcal{V}^1((\partial_{s_2}\eta^S)^2+1)-\mathcal{V}^2(\partial_{s_1}\eta^S\partial_{s_2}\eta^S)}{| N|^2 }, \\
G_2&=&\frac{\mathcal{V}^2((\partial_{s_1}\eta^S)^2+1)-\mathcal{V}^1(\partial_{s_1}\eta^S\partial_{s_2}\eta^S)}{| N|^2 },\\
G_3&=&\frac{\partial_{s_1}\eta^S\mathcal{V}^1+\partial_{s_2}\eta^S\mathcal{V}^2}{| N|^2},
\end{array}
\end{equation}
and $G$ is defined on $\Gamma_0$ by $G=\sum_{i=1}^2\mathcal{V}^i\tau^i$.


We  set
$$
u=\widetilde{u}-w^S,\quad p=\widetilde{p}-p^S,\quad \xi=\eta-\eta^S.
$$
Then, $(u,p,\xi)$ satisfies the system
\begin{equation}\label{nav3}
\left\{
\begin{array}{rl}
\partial_t u-\nabla \cdot \mathbb{T}(u,p)+(w^S\cdot\nabla)u+(u\cdot \nabla) w^S =F(u,p,\xi)  & \text{in}\ (0,\infty)\times \Omega,\\
\nabla \cdot u=0  & \text{in}\ (0,\infty)\times \Omega,\\
\partial_{tt} \xi+A_1\xi +A_2 \partial_t \xi=\mathbb{H}_{\eta^S}(u,p)+H(u+w^S,\xi+\eta^S) &\text{in}\ (0,\infty)\times \omega, \\
\end{array}
\right.
\end{equation}
where
\begin{multline}
\label{defFbar}
F(u,p,\xi)=(\Id-\mathbf{K}_{\xi+\eta^S})\partial_t u
-\nu(\Delta-\mathbf{L}_{\xi+\eta^S})(u+w^S)+(\nabla-\mathbf{G}_{\xi+\eta^S})(p+p^S) \\+\mathbf{M}_{\xi+\eta^S}(u+w^S)+ \mathbf{N}_{\xi+\eta^S}(u+w^S)+(u\cdot\nabla) w^S+(w^S\cdot\nabla) u+(w^S\cdot\nabla) w^S\\ +\det(\nabla X)f^S(X)-f^S, 
\end{multline}
with the boundary conditions
\begin{equation}\label{nav3b}
\left\{
\begin{array}{rl}
\left[u-\mathcal{T}\partial_t \xi\right]_{n}=1_{\Gamma_0}(\mathbb{M}v)_{n} &  \text{on}\ (0,\infty)\times \partial\Omega,\\
\left[2\nu D(u)n+\beta (u-\mathcal{T}\partial_t \xi)\right]_{\tau}=1_{\Gamma_0}\beta(\mathbb{M}v)_{\tau}+G(u+w^S,\xi+\eta^S) &  \text{on}\ (0,\infty)\times \partial\Omega,
\end{array}
\right. 
\end{equation}
and the initial conditions
\begin{equation}\label{nav3c}
\left\{
\begin{array}{rl}
u(0,\cdot )=\widetilde{u}^0-w^S  &\text{in}\ \Omega,\\
\xi(0, \cdot)=\eta^0-\eta^S &\text{in}\  \omega,\\
\partial_t \xi(0,\cdot)=\eta^1&\text{in}\  \omega.\\
\end{array}
\right.
\end{equation}
\begin{Remark}
\label{r}
For the linearization of the system \eqref{nav3}, \eqref{nav3b}, \eqref{nav3c},
we need to find all the linear terms coming from the change of variables. For example, we observe the term
$$
\det(\nabla X)f^S(X)-f^S=\det(\nabla X)(f^S(X)-f^S)+(\det(\nabla X)-1)f^S.
$$
We recall that $\det(\nabla X)=\frac{1+\eta}{1+\eta^S}$.
We notice that
$$
X(y)=y+\frac{ y_3\xi(y_1,y_2 )}{1+\eta^S(y_1,y_2)} e_3=y+\theta\xi(y_1,y_2) e_3.
$$
We apply Taylor's formula since $f^S\in [W^{2,\infty}(\Omega)]^3$, we get	
\begin{multline}
f^S(X(y))-f^S(y)=\theta\xi(y_1,y_2 ) \nabla f^S(y)e_3\\+\int_0^1(1-s)\nabla^2f^S(y+s\theta\xi(y_1,y_2 ) e_3)\theta\xi(y_1,y_2 ) e_3\cdot\theta\xi(y_1,y_2 ) e_3 \ ds.
\end{multline}
Then
\begin{multline}
\label{16:33}
\det(\nabla X(y))f^S(X(y))-f^S(y)\\=\frac{\xi(y_1,y_2)}{1+\eta^S(y_1,y_2)}f^S(y)+\theta\xi(y_1,y_2) \nabla f^S(y)e_3+\frac{(\xi(y_1,y_2))^2}{1+\eta^S}\theta \nabla f^S(y)e_3\\+\left(1+\frac{\xi(y_1,y_2)}{1+\eta^S(y_1,y_2)}\right) \int_0^1(1-s)\nabla^2f^S(y+s\theta\xi(y_1,y_2 ) e_3)\theta\xi(y_1,y_2 ) e_3\cdot\theta\xi(y_1,y_2 ) e_3 \ ds.
\end{multline}
Now, let assume that $\left\|\xi \right\|_{L_\gamma^2(0,\infty;L^2(\omega))}\leq CR  $, then, we have for example
$$
\left\|\frac{\xi}{1+\eta^S}f^S\right\|_{L_\gamma^2(0,\infty;[L^2(\Omega)]^3)}\leq CR. 
$$
Thus, the terms of this kind should be considered in the linear part of the system or else it will constitute a problem in the fixed point procedure. 
	
To overcome this difficulty, we follow the same approach described in \cite{MT}. More precisely, we write for example
$$
\det(\nabla X(y))f^S(X(y))-f^S(y)=\ell^{(12)}(\xi)+\epsilon^{(12)}(\xi),
$$ 
with 
\begin{equation}
\left\|\epsilon^{(12)}(\xi)\right\|_{L^2_\gamma(0,_\infty; [L^2(\Omega)]^3)}\leq C\left\|\xi\right\|_{L^2_\gamma(0,_\infty; L^2(\omega))}^2(1+\left\|\xi\right\|_{L^2_\gamma(0,_\infty; L^2(\omega))}^m),\quad m\geq 1.  
\end{equation}
where $\ell^{(12)}$ is the linear operator
$$
\ell^{(12)}(\xi)=\frac{\xi}{1+\eta^S}f^S,
$$
and $\epsilon^{(12)}$ is the remaining non linear terms appearing in \eqref{16:33}. Then, the strategy consists to integrate the linear operator $\ell^{(12)}$ in the left hand side of the fluid equation in \eqref{nav3} and to inject $\epsilon^{(12)}$ in the right hand side that will be handled in the fixed point section. 
\end{Remark}
Due to \cref{r}, we need to decompose the functions $F$, $H$ and $G$ into a linear part in $\xi$ and a nonlinear part.

The notation $\partial_s\xi$ (resp. $\partial^2_{ss}\xi$, $\partial^3_{sss}\xi$),  designates in what follows the first derivative $\partial_{s_j}\xi$ (resp. the second derivative $\partial^2_{s_is_j}\xi$, the third derivative $\partial^3_{s_is_js_k}\xi$ ).

We underline that the linear part is not given in a explicit way, we just need to write formally the terms that appear precising the derivative order of $\xi$. The linear parts denoted usually by  $\ell^{(i)}$ will have the form that formally writes 
\begin{equation}
\label{17:17}
\ell^{(i)}(Z_1,...,Z_m)=P^{(i)}_1Z_1+....+P^{(i)}_mZ_m,
\end{equation}
where $P^{(i)}_j$ are coefficients that depend only on the stationary state $(w^S,p^S,\eta^S)$, consequently, they are regular enough. The goal now is to decompose each terms appearing in the formulas of $F(u,p,\xi)$, $H(u+w^s,\xi+\eta^S)$, $G(u+w^S,\xi+\eta^S)$ as it is done in \cref{r}, the formal calculus are done in \cref{fff1-}. Note that since our system is coupled, we need to obtain a convenient match between the linear operators that appear in the fluid equation, the structure equation and even in the boundary conditions. That is given in the following lemma.  
\begin{Lemma} Suppose that
$$	
(w^S,p^S,\eta^S)\in [W^{2,\infty}(\Omega)]^3\times W^{1,\infty}(\Omega)/ \mathbb{R}\times  C^4(\omega).
$$
Then,
\begin{itemize}
\item there exist linear operators $$\mathcal{L}^1(\cdot)\in  \mathcal{L}(\mathcal{D}(A_1^{3/4}),[H^1(\Omega)]^{3\times3}),\quad\mathcal{L}^2(\cdot,\cdot)\in\mathcal{L}(\mathcal{D}(A_1^{1/2})\times\mathcal{D}(A_1^{1/4}),[L^2(\Omega)]^3),$$ 
$$
\mathcal{L}^3(\cdot)\in \mathcal{L}(\mathcal{D}(A_1^{1/2}),[H^{1/2}(\partial\Omega)]^3),
$$
\item there exist non linear functions $\mathcal{N}_E$, $\mathcal{N}_F$,  $\mathcal{N}_G$ and $\mathcal{F}$,
\end{itemize}
 such that the system \eqref{nav3}, \eqref{nav3b}, \eqref{nav3c} is equivalent to
\begin{equation}\label{nav4}
\left\{
\begin{array}{rl}
\partial_t u-\nabla \cdot \mathbb{T}(u,p)+(w^S\cdot\nabla)u+(u\cdot \nabla) w^S -\nabla\cdot\mathcal{L}^1(\xi)+\mathcal{L}^{2}(\xi,\partial_{t}\xi)\\=\nabla\cdot\mathcal{N}_E(\xi)+\mathcal{N}_F(\xi)+\mathcal{F}(u,p,\xi)& \text{in}\ (0,\infty)\times \Omega,\\
\nabla \cdot u=0  & \text{in}\ (0,\infty)\times \Omega,\\
\partial_{tt} \xi+A_1\xi +A_2 \partial_t \xi+\mathcal{T}^*(\mathcal{L}^1(\xi)n)=\mathbb{H}_{\eta^S}(u,p)+H(u,\xi+\eta^S)\\-\mathcal{T}^*(\mathcal{N}_E(\xi)n) &\text{in}\ (0,\infty)\times \omega, \\
\end{array}
\right.
\end{equation}
with the boundary conditions
\begin{equation}\label{nav4b}
\left\{
\begin{array}{rl}
\left[u-\mathcal{T}\partial_t \xi\right]_{n}=1_{\Gamma_0}(\mathbb{M}v)_{n} &  \text{on}\ (0,\infty)\times \partial\Omega,\\
\left[2\nu D(u)n+\mathcal{L}^1(\xi)n+\mathcal{L}^3(\xi)+\beta (u-\mathcal{T}\partial_t \xi)\right]_{\tau}=1_{\Gamma_0}\beta(\mathbb{M} v)_{\tau}\\-[\mathcal{N}_E(\xi)n]_{\tau}+[\mathcal{N}_G(\xi)]_{\tau}+G(u,\xi+\eta^S) &  \text{on}\ (0,\infty)\times \partial\Omega,
\end{array}
\right. 
\end{equation}
and the initial conditions
\begin{equation}\label{nav4c}
\left\{
\begin{array}{rl}
u(0,\cdot )=\widetilde{u}^0-w^S  &\text{in}\ \Omega,\\
\xi(0, \cdot)=\eta^0-\eta^S &\text{in}\  \omega,\\
\partial_t \xi(0,\cdot)=\eta^1&\text{in}\  \omega.\\
\end{array}
\right.
\end{equation}
\end{Lemma}
\begin{proof}
\textbf{Decomposition of $F(u,p,\xi)$:}\\
Let decompose the term $-\nu(\Delta-\mathbf{L}_\eta)$.
We notice that
\begin{equation}
\label{yes1}	
-\nu(\Delta-\mathbf{L}_\eta)(\widetilde{u})=-\nu(\Delta-\mathbf{L}_{\xi+\eta^S})(u+w^S)=-\nu(\Delta-\mathbf{L}_{\xi+\eta^S})(u)-\nu(\Delta-\mathbf{L}_{\xi+\eta^S})(w^S).
\end{equation}
The term $-\nu(\Delta-\mathbf{L}_{\xi+\eta^S})(u)$ is nonlinear whereas $-\nu(\Delta-\mathbf{L}_{\xi+\eta^S})(w^S)$ contains linear terms in $\xi$, then, from \cref{r}, we should write
\begin{equation}
\label{m1}
-\nu(\Delta-\mathbf{L}_{\xi+\eta^S})(w^S)= \ell^{(13)}(\xi,\partial_{s}\xi,\partial^2_{ss}\xi,\partial^3_{sss}\xi)+\epsilon^{(13)}(\xi,\partial_{s}\xi,\partial^2_{ss}\xi,\partial^3_{sss}\xi),
\end{equation}
where $\ell^{(13)}$ is a linear operator of the form \eqref{17:17} and $\epsilon^{(13)}$ is the nonlinear part. To write the linearized system in an appropriate form, we define for all $\widetilde{u}$
\begin{multline}
\label{defE}
(E_\eta(\widetilde{u}))_{im}=
\nu\sum_{j,k,l}\left( \det(\nabla X)a_{ik}(X)\frac{\partial Y_m}{\partial x_j}(X)\frac{\partial Y_l}{\partial x_j}(X)-\delta_{ik}\delta_{mj}\delta_{jl}\right) 
\frac{\partial \widetilde{u}_k}{\partial y_l}
\\
+\nu\sum_{j,k,l}\left( \det(\nabla X)a_{jk}(X)\frac{\partial Y_m}{\partial x_i}(X)\frac{\partial Y_l}{\partial x_j}(X)-\delta_{jk}\delta_{mi}\delta_{jl}\right) 
\frac{\partial \widetilde{u}_k}{\partial y_l}
\\
+\det(\nabla X)\left(\nu\sum_{j,k}\left(\frac{\partial a_{ik}}{\partial x_j}(X)\frac{\partial Y_m}{\partial x_j}(X)+\frac{\partial a_{jk}}{\partial x_i}(X)\frac{\partial Y_m}{\partial x_j}(X)\right) \widetilde{u}_k\right).
\end{multline}
Observe that 
\begin{equation}
\label{t6m}
(\nabla\cdot E_\eta)_i=-\nu[(\Delta-\mathbf{L}_\eta)]_i+(F_\eta^1)_i,
\end{equation}
with
\begin{multline}
[(F_\eta^1)(\widetilde{u})]_i=\nu\frac{\partial\det(\nabla X)}{\partial y_m}\sum_{j,k,l,m}\Bigg[\left( a_{ik}(X)\frac{\partial Y_m}{\partial x_j}(X)\frac{\partial Y_l}{\partial x_j}(X)+a_{jk}(X)\frac{\partial Y_m}{\partial x_i}(X)\frac{\partial Y_l}{\partial x_j}(X)\right) 
\frac{\partial\widetilde{u}_k}{\partial y_l}
\\
+\sum_{j,k}\left(\frac{\partial a_{ik}}{\partial x_j}(X)\frac{\partial Y_m}{\partial x_j}(X)+\frac{\partial a_{jk}}{\partial x_i}(X)\frac{\partial Y_m}{\partial x_j}(X)\right) \widetilde{u}_k
\Bigg]
\\
+\nu\det(\nabla X)\Bigg[
\sum_{j,k,l,m,q}\frac{\partial^2 Y_m}{\partial x_j\partial x_q}(X)\frac{\partial X_q}{\partial y_m}\bigg(\left( a_{ik}(X)\frac{\partial Y_l}{\partial x_j}(X)+a_{jk}(X)\frac{\partial Y_l}{\partial x_j}(X)\right)\frac{\partial \widetilde{u}_k}{\partial y_l} \\+ \left( \frac{\partial a_{ik}}{\partial x_j}(X)+\frac{\partial a_{jk}}{\partial x_i}(X)\right) \widetilde{u}_k\bigg)\Bigg].
\end{multline}
From the \cref{fff1-}, we have
\begin{equation}
\label{E}
(E_{\xi+\eta^S}(w^S))=\mathcal{L}^1(\xi)+\mathcal{N}_E(\xi),
\end{equation}
where $\mathcal{L}^1$ is a linear operator that writes
$$
\mathcal{L}^1(\xi)=Q^1\xi+Q^2\partial_s\xi+Q^3\partial^2_{ss}\xi,
$$
where $Q^i$, $i=1,..,3$ are operators that depend only on the stationary state and are of regularity $W^{1,\infty}(\Omega)$. The quantity $\mathcal{N}_E(\xi)$ is the remaining nonlinear part that contains the terms of the form
\begin{equation}
\label{ter*}
c^1(w^S,\eta^S)\frac{\xi^{m_1}(\partial_s\xi)^{m_2}}{(1+\xi+\eta^S)^{\alpha_1}},\ c^2(w^S,\eta^S)\frac{\xi^{n_1}(\partial_s\xi)^{n_2}\partial_{ss}^2\xi}{(1+\xi+\eta^S)^{\alpha_2}},\quad m_1+m_2\geq 2,\ n_1+n_2\geq 1, \ \alpha_i\geq 1,
\end{equation}
where $c^1(w^S,\eta^S)$ and $c^2(w^S,\eta^S)$ are two quantities of regularity $W^{1,\infty}(\Omega)$.
In the other hand, from \cref{fff1-}, we have also
\begin{equation}
\label{zak1}
(F_{\xi+\eta^S}^1)(w^S)=-\ell^{(9)}(\xi, \partial_s\xi,\partial_{ss}^2\xi)-\epsilon^{(9)}(\xi,\partial_s\xi,\partial_{ss}^2\xi),
\end{equation}
where the expression of the nonlinear part $\epsilon^{(9)}(\xi,\partial_s\xi,\partial_{ss}^2\xi)$ contains terms of the form \eqref{ter*}.
Then, from \eqref{m1}, \eqref{t6m}, \eqref{E} and \eqref{zak1}, we deduce
\begin{equation}
\label{l13}
\ell^{(13)}(\xi,\partial_{s}\xi,\partial^2_{ss}\xi,\partial^3_{sss}\xi)=\nabla\cdot(\mathcal{L}^1(\xi))+\ell^{(9)}(\xi,\partial_s\xi,\partial_{ss}^2\xi),
\end{equation}
\begin{equation}
\label{e13}
\epsilon^{(13)}(\xi,\partial_{s}\xi,\partial^2_{ss}\xi,\partial^3_{sss}\xi)=\epsilon^{(9)}(\xi,\partial_s\xi,\partial_{ss}^2\xi)+\nabla\cdot(\mathcal{N}_E(\xi)).
\end{equation}
Using \eqref{yes1}, \eqref{m1}, \eqref{l13} and \eqref{e13}, the term $-\nu(\Delta-\mathbf{L}_{\xi+\eta^S})(u+w^S)$ writes
\begin{multline}
\label{nicki1}
-\nu(\Delta-\mathbf{L}_{\xi+\eta^S})(u+w^S)=\nabla\cdot(\mathcal{L}^1(\xi))+\ell^{(9)}(\xi,\partial_s\xi,\partial_{ss}^2\xi)
+\epsilon^{(9)}(\xi,\partial_s\xi,\partial_{ss}^2\xi)+\nabla\cdot(\mathcal{N}_E(\xi))\\-\nu(\Delta-\mathbf{L}_{\xi+\eta^S})(u).
\end{multline}
Furthermore, using again \cref{fff1-}, we obtain
\begin{equation}
\label{nicki2}
[(\nabla-\mathbf{G}_{\xi+\eta^S})(p^S)]_i=\sum_k(\delta_{ki}-\det(\nabla X)\frac{\partial Y_k}{\partial x_i}(X))\frac{\partial p^S}{\partial y_k}=\ell^{(5)}_i(\xi,\partial_s\xi)+\epsilon_i^{(5)}(\xi,\partial_s\xi),
\end{equation}
where $\epsilon^{(5)}(\xi,\partial_s\xi)$ contains the terms 
\begin{equation}
\label{ep2*}
c^3(p^S,\eta^S)\frac{\xi^{m_2}(\partial_s\xi)^{m_2}}{(1+\xi+\eta^S)^{\alpha_1}},\quad m_1+m_2\geq 2,\ \alpha_1\geq 1,
\end{equation}
where $c^3$ is a function that depends on the stationary state and bounded in $\Omega$.
Then, we get
\begin{equation}
\label{18:42}
[(\nabla-\mathbf{G}_{\xi+\eta^S})(p+p^S)]_i=\ell_i^{(5)}(\xi,\partial_s\xi)+\epsilon_i^{(5)}(\xi,\partial_s\xi)+[(\nabla-\mathbf{G}_{\xi+\eta^S})(p)]_i.
\end{equation}
Moreover, we have
$$
\sum_{l,k}\det(\nabla X)a_{ik}(X)\frac{\partial w^S_k}{\partial y_l}\partial_tY_l(X)=(\nabla X\nabla w^S\partial_tY(X))_i,
$$
where
\begin{multline*}
(\nabla X\nabla w^S\partial_tY(X))_i=\partial_tY_3\frac{\partial w^S_i}{\partial y_3}, \ i=1,2,\\ (\nabla X\nabla w^S\partial_tY(X))_3=\partial_tY_3\left( \frac{\partial X_3}{\partial y_1}\frac{\partial w^S_1}{\partial y_3}+ \frac{\partial X_3}{\partial y_2}\frac{\partial w^S_2}{\partial y_3}+\frac{1+\eta}{1+\eta^S}\frac{\partial w^S_3}{\partial y_3}\right).
\end{multline*}
We notice that
$$
\partial_tY_3=-y_3\frac{\partial_t\xi}{1+\eta},\quad \frac{\partial X_3}{\partial y_i}=y_3\frac{\partial_{s_i}\xi(1+\eta^S)-\partial_{s_i}\eta^S\xi}{(1+\eta^S)^2}, \ i=1,2.
$$
Using \eqref{17:37}, we find
$$
-(\nabla X\nabla w^S\partial_tY(X))_i=\ell_i^{(4)}(\partial_t\xi)+\epsilon_i^{(4)}(\xi,\partial_s\xi,\partial_t\xi),
$$
where $\epsilon^{(4)}(\xi,\partial_s\xi,\partial_t\xi)$ has the terms
\begin{equation}
c^4(w^S,\eta^S)\frac{ \xi^{m_1}\partial_t\xi\partial_s\xi}{(1+\xi+\eta^S)^{\alpha_1}},\ c^5(w^S,\eta^S)\frac{\xi^{m_2}\partial_t\xi}{(1+\xi+\eta^S)^{\alpha_2}},\quad m_1\geq 0,\quad m_2\geq 1, \ \alpha_i\geq 1,
\end{equation}
with $c^4$ and $c^5$ are $W^{1,\infty}(\Omega)$ functions. Moreover, using \cref{fff1-}, we get
\begin{equation}
-\sum_k\det(\nabla X)\partial_t a_{ik}(X)w^S_k=\ell_i^{(6)}(\partial_t\xi,\partial^2_{ts}\xi)+\epsilon_i^{(6)}(\xi,\partial_t\xi,\partial^2_{ts}\xi),
\end{equation}
where $\epsilon^{(6)}(\xi,\partial_t\xi,\partial^2_{ts}\xi)$ contains terms of the form
\begin{multline}
c^6(w^S,\eta^S)\frac{\xi^{m_1}\partial^2_{ts}\xi}{(1+\xi+\eta^S)^{\alpha_1}},\ c^7(w^S,\eta^S)\frac{\xi^{m_2}\partial_t\xi}{(1+\xi+\eta^S)^{\alpha_2}} ,\  c^8(w^S,\eta^S)\frac{\xi^{m_3}\partial_t\xi\partial_s\xi}{(1+\xi+\eta^S)^{\alpha_3}},\\ m_1\geq 1,\ m_2 \geq 1,\ m_3\geq 0,\ \alpha_i\geq 1,
\end{multline}
where $c^i$, $i=6,7,8$ are $W^{2,\infty}(\Omega)$ functions.
Then
\begin{equation}
\label{nicki3}
\mathbf{M}_{\xi+\eta^S}(u+w^S)=\ell^{(4)}(\partial_t\xi)+\ell^{(6)}(\partial_t\xi,\partial^2_{ts}\xi)+\epsilon^{(6)}(\xi,\partial_t\xi,\partial^2_{ts}\xi)+\epsilon^{(4)}(\xi,\partial_s\xi,\partial_t\xi)+\mathbf{M}_{\xi+\eta^S}(u).
\end{equation}
We have, 
\begin{multline}
[\mathbf{N}_{\xi+\eta^S}(u+w^S)+u\cdot\nabla w^S+w^S\cdot\nabla u+w^S\cdot\nabla w^S]_i=-\sum_{k,l,j}\det(\nabla X)a_{kl}(X)\frac{\partial a_{ij}(X)}{\partial x_k}w^S_lw^S_j
\\
+\sum_{k,l,j,m}\left( \delta_{ij}\delta_{kl}\delta_{km}-\det(\nabla X)a_{kl}(X)a_{ij}(X)\frac{\partial Y_m}{\partial x_k}(X)\right)w^S_l\frac{\partial w^S_j}{\partial y_m}+F^2_i(u,\xi+\eta^S)\\=F^2_i(u,\xi+\eta^S)+F_i^3(w^S,\xi+\eta^S)
,\quad i=1,2,3.
\end{multline}
Using \cref{fff1-}, we obtain
\begin{equation}
F^3_i(w^S,\xi+\eta^S)=\ell_i^{(8)}(\xi,\partial_s\xi,\partial^2_{ss}\xi)+\epsilon_i^{(8)}(\xi,\partial_s\xi,\partial^2_{ss}\xi),
\end{equation}
where $\epsilon^{(8)}(\xi,\partial_s\xi,\partial^2_{ss}\xi)$ admits terms of the form \eqref{ter*}.
Then, we get 
\begin{multline}
\label{nicki4}
[\mathbf{N}_{\xi+\eta^S}(u+w^S)+u\cdot\nabla w^S+w^S\cdot\nabla u+w^S\cdot\nabla w^S]_i=\ell_i^{(8)}(\xi,\partial_s\xi,\partial^2_{ss}\xi)\\+\epsilon_i^{(8)}(\xi,\partial_s\xi,\partial^2_{ss}\xi)+F^2_i(u,\xi+\eta^S),\quad i=1,2,3.
\end{multline}
From \eqref{defFbar}, \eqref{nicki1}, \eqref{18:42}, \eqref{nicki3} and \eqref{nicki4}, we get
\begin{equation}
F(u,p,\xi)=\nabla\cdot\mathcal{L}^{1}(\xi)+\mathcal{L}^2(\xi,\partial_t\xi)+\nabla\cdot\mathcal{N}_E(\xi)+\mathcal{N}_F(\xi)+\mathcal{F}(u,p,\xi).
\end{equation}
where
\begin{equation}
\mathcal{L}^2(\xi,\partial_t\xi)=\ell^{(9)}(\xi,\partial_s\xi,\partial^2_{ss}\xi)+\ell^{(5)}(\xi,\partial_s\xi)+\ell^{(4)}(\partial_t\xi)+\ell^{(6)}(\partial_t\xi,\partial^2_{ts}\xi)+\ell^{(8)}(\xi,\partial_s\xi,\partial^2_{ss}\xi).
\end{equation}
More precisely, 
$$
\mathcal{L}^2(\xi,\partial_t\xi)=Q^4\xi+Q^5\partial_s\xi+Q^6\partial^2_{ss}\xi+Q^7\partial_t\xi+Q^8\partial^2_{st}\xi,
$$
where $Q_i$, $i=4,..,8$ are operators of regularity $W^{1,\infty}(\Omega)$ and we also have
\begin{multline}
\label{nf}
\mathcal{N}_F(\xi)=\epsilon^{(9)}(\xi,\partial_s\xi,\partial^2_{ss}\xi)+\epsilon^{(5)}(\xi,\partial_s\xi)+\epsilon^{(4)}(\xi,\partial_s\xi,\partial_t\xi)+\epsilon^{(6)}(\xi,\partial_t\xi,\partial^2_{ts}\xi)+\epsilon^{(8)}(\xi,\partial_s\xi,\partial^2_{ss}\xi),
\end{multline}
\begin{multline}
\label{fcall}
\mathcal{F}(u,p,\xi)=(\Id-\mathbf{K}_{\xi+\eta^S})\partial_t u-\nu(\Delta-\mathbf{L}_{\xi+\eta^S})(u)+(\nabla-\mathbf{G}_{\xi+\eta^S})(p)+\mathbf{M}_{\xi+\eta^S}(u)+F^2(u,\xi+\eta^S).
\end{multline}

\textbf{Decomposition of $H(u,\xi+\eta^S)$:}\\
We have
$$
H(u+w^S,\xi+\eta^S)=H(u,\xi+\eta^S)+H(w^S,\xi+\eta^S).
$$
We notice that
\begin{equation}
H(w^S,\xi+\eta^S)=-\mathcal{T}^*((E_{\xi+\eta^S}(w^S))n),
\end{equation}
where $\mathcal{T}^*$ is defined by
\begin{equation}
\label{t*}
\mathcal{T}^*(\zeta)(s)=M\left( \sqrt{1+\left|\nabla \eta^S\right|^2 }\zeta(s,1+\eta^S(s))\cdot e_3\right) ,\quad \forall \zeta\in [L^2(\partial\Omega)]^3.
\end{equation}
Using \eqref{E}, we obtain 
\begin{equation}
H(u+w^S,\xi+\eta^S)=-\mathcal{T}^*(\mathcal{L}^1(\xi)n)+H(u,\xi+\eta^S)-\mathcal{T}^*(\mathcal{N}_E(\xi)n).
\end{equation}

\textbf{Decomposition of $G(u,\xi+\eta^S)$:}\\
Using the expression of $\mathcal{W}$ in \eqref{tak4.1*} and of $\mathcal{V}^i$ in \eqref{tak4.2S}, we find
\begin{multline*}
\mathcal{V}^i(w^S,\xi+\eta^S)=
\left( 2\nu D(w^S)n+\beta w^S\right)\cdot(\tau^i-\widetilde{\tau}^i(X))
\\+\Bigg[ \nu\sum_{j,m,q} \left( (n)_j\delta_{km}\delta_{qj}
-(\widetilde{n}(X))_ja_{km}(X)\frac{\partial Y_q}{\partial x_j}(X)\right)\frac{\partial w^S_m}{\partial y_q}
\\
-\nu\sum_{j,m,q} \left( 
(\widetilde{n}(X))_ja_{jm}(X)\frac{\partial Y_q}{\partial x_k}(X)-(n)_j\delta_{jm}\delta_{qk}\right)\frac{\partial w^S_m}{\partial y_q}
\\
- \nu\sum_{j,m} (\widetilde{n}(X))_j\left( \frac{\partial a_{km}}{\partial x_j}(X)w^S_m+\frac{\partial a_{jm}}{\partial x_k}(X)w^S_m\right)
+\beta \sum_j (\delta_{kj}-a_{kj}(X)) w^S_j\Bigg]\widetilde{\tau}^i_k(X),
\quad k=1,2,3.
\end{multline*}
We notice that
\begin{multline}
\nu\sum_{j,m,q} \left( n_j\delta_{km}\delta_{qj}
-\widetilde{n}_j(X)a_{km}(X)\frac{\partial Y_q}{\partial x_j}(X)\right)\frac{\partial w^S_m}{\partial y_q}
\\
+\nu\sum_{j,m,q} \left(n_j\delta_{jm}\delta_{qk}- 
\widetilde{n}_j(X)a_{jm}(X)\frac{\partial Y_q}{\partial x_k}(X)\right)\frac{\partial w^S_m}{\partial y_q}
\\
- \nu\sum_{j,m} \widetilde{n}_j(X)\left( \frac{\partial a_{km}}{\partial x_j}(X)w^S_m+\frac{\partial a_{jm}}{\partial x_k}(X)w^S_m\right)
=-\frac{|N|}{|\widetilde{N}|}\left[(E_{\xi+\eta^S}(w^S))n\right]_k\\+ \nu\sum_{j} \left( (n)_j
-\frac{N_j}{|\widetilde{N}|}\right)\frac{\partial w^S_k}{\partial y_j}
+\nu\sum_{j} \left( (n)_j
-\frac{N_j}{|\widetilde{N}|}\right)\frac{\partial w^S_j}{\partial y_k}. 
\end{multline}
In the other hand, using Taylor's formula
\begin{equation*}
\frac{1}{|\widetilde{N}|}=\frac{1}{|N|}-\frac{\nabla\eta^S\cdot\nabla\xi}{|N|^3}-\frac{|\nabla \xi|^2}{2|N|^3}+\frac{3\theta^2}{4|N|}\int_0^1\frac{(1-s)}{(1+s\theta)^{5/2}}\ ds,
\end{equation*}
where $\theta=\dfrac{2\nabla\eta^S\cdot\nabla\xi+|\nabla \xi|^2}{|N|^2}$. Then, we deduce

\begin{equation}
\label{rec}
\frac{1}{|\widetilde{N}|}=\frac{1}{|N|}-\frac{\nabla\eta^S\cdot\nabla\xi}{|N|^3}+\epsilon^{(10)}(\partial_s\xi),
\end{equation}
with
$$
\epsilon^{(10)}(\partial_s\xi)=-\frac{|\nabla \xi|^2}{2|N|^3}+\frac{3\theta^2}{4|N|}\int_0^1\frac{(1-s)}{(1+s\theta)^{5/2}}\ ds.
$$
Then, using \eqref{E} and \eqref{rec}, we obtain
\begin{multline}
\label{11:13}
\nu\sum_{j,m,q} \left( n_j\delta_{km}\delta_{qj}
-\widetilde{n}_j(X)a_{km}(X)\frac{\partial Y_q}{\partial x_j}(X)\right)\frac{\partial w^S_m}{\partial y_q}
\\
+\nu\sum_{j,m,q} \left(n_j\delta_{jm}\delta_{qk}- 
\widetilde{n}_j(X)a_{jm}(X)\frac{\partial Y_q}{\partial x_k}(X)\right)\frac{\partial w^S_m}{\partial y_q}
\\
- \nu\sum_{j,m} \widetilde{n}_j(X)\left( \frac{\partial a_{km}}{\partial x_j}(X)w^S_m+\frac{\partial a_{jm}}{\partial x_k}(X)w^S_m\right)
=[-\mathcal{L}^1(\xi)n+\ell^{(14)}(\partial_s\xi)-\mathcal{N}_E(\xi)n\\+\epsilon^{(14)}(\xi,\partial_s\xi,\partial^2_{ss}\xi)]_k,
\end{multline}
where
$$
\ell_k^{(14)}(\partial_s\xi)=\nu\frac{\nabla\eta^S\cdot\nabla\xi}{|N|^2}\sum_{j}\left(\frac{\partial w^S_k}{\partial y_j}+\frac{\partial w^S_j}{\partial y_k} \right), 
$$
and $\epsilon^{(14)}$ admits the terms \eqref{ter*}. Moreover, From \cref{fff1-}, we get
$$
\beta \sum_j (\delta_{kj}-a_{kj}(X)) w^S_j=\ell_{k}^{(15)}(\xi,\partial_{s}\xi)+\epsilon_{k}^{(15)}(\xi,\partial_{s}\xi),
$$
where $\epsilon_{k}^{(15)}$ had the terms
\begin{equation*}
c^9(w^S,\eta^S)\frac{\xi^{m_1}}{(1+\xi+\eta^S)^{\alpha_1}},\quad c^{(10)}(w^S,\eta^S)\frac{\xi ^{m_2}\partial_s\xi}{(1+\xi+\eta^S)^{\alpha_2}}, \ m_1\geq 2,\ m_2 \geq 1,\ \alpha_i\geq 1,\ i=1,2,
\end{equation*}
where $c^i$, $i=9,10$ are of regularity $W^{3/2,\infty}(\partial\Omega)$.
\begin{equation}
\mathcal{V}^i(w^S,\eta)=-[\mathcal{L}^1(\xi)n+\mathcal{L}^3(\xi)]\cdot\tau^i-[\mathcal{N}_E(\xi)n]\cdot \tau^i+\epsilon_i^{(11)}(\xi,\partial_s\xi,\partial^2_{ss}\xi),
\end{equation}
where  $\mathcal{L}^3$ is an operator defined by
\begin{equation*}
\mathcal{L}^3(\xi)\cdot \tau^i=-\left( 2\nu D(w^S)n+\beta w^S\right)\cdot(\tau^i-\widetilde{\tau}^i(X))-\ell^{(14)}(\partial_s\xi)\cdot \tau^i,\quad \mathcal{L}^3(\xi)\cdot n=0,
\end{equation*}
then $\mathcal{L}^3$ can be expressed as the following 
\begin{equation}
\label{l3*}
\mathcal{L}^3(\xi)=Q^9\xi+Q^{10}\partial_s\xi,
\end{equation}
where $Q^9$ and $Q^{10}$ are of regularity $W^{1/2,\infty}(\partial\Omega)$ and	
\begin{multline*}
\epsilon^{(11)}(\xi,\partial_s\xi,\partial^2_{ss}\xi)=-\ell^{(14)}(\partial_s\xi)\cdot (\widetilde{ \tau}^i(X)-\tau^i)+\epsilon^{(14)}(\xi,\partial_s\xi,\partial^2_{ss}\xi)\cdot \widetilde{ \tau}^i(X)-[\mathcal{N}_E(\xi)n]\cdot(\widetilde{ \tau}^i(X)- \tau^i).
\end{multline*}
Here, $\epsilon^{(11)}$ contains terms of the type \eqref{ter*}. Finally, we define $\mathcal{N}_G$ such that
$$\mathcal{N}_G(\xi)\cdot \tau^i=\epsilon_i^{(11)}(\xi,\partial_s\xi,\partial^2_{ss}\xi),\quad \mathcal{N}_G(\xi)\cdot n=0.$$

Then, the system \eqref{nav3}, \eqref{nav3b}, \eqref{nav3c} becomes
\begin{equation*}\label{nav4*}
\left\{
\begin{array}{rl}
\partial_t u-\nabla \cdot \mathbb{T}(u,p)+(w^S\cdot\nabla)u+(u\cdot \nabla) w^S -\nabla\cdot\mathcal{L}^1(\xi)+\mathcal{L}^{2}(\xi,\partial_{t}\xi)\\=\nabla\cdot\mathcal{N}_E(\xi)+\mathcal{N}_F(\xi)+\mathcal{F}(u,p,\xi)& \text{in}\ (0,\infty)\times \Omega,\\
\nabla \cdot u=0  & \text{in}\ (0,\infty)\times \Omega,\\
\partial_{tt} \xi+A_1\xi +A_2 \partial_t \xi+\mathcal{T}^*(\mathcal{L}^1(\xi)n)=\mathbb{H}_{\eta^S}(u,p)+H(u,\xi+\eta^S)\\-\mathcal{T}^*(\mathcal{N}_E(\xi)n) &\text{in}\ (0,\infty)\times \omega, \\
\end{array}
\right.
\end{equation*}
with the boundary conditions
\begin{equation*}\label{nav4b*}
\left\{
\begin{array}{rl}
\left[u-\mathcal{T}\partial_t \xi\right]_{n}=1_{\Gamma_0}(\mathbb{M}v)_{n} &  \text{on}\ (0,\infty)\times \partial\Omega,\\
\left[2\nu D(u)n+\mathcal{L}^1(\xi)n+\mathcal{L}^3(\xi)+\beta (u-\mathcal{T}\partial_t \xi)\right]_{\tau}=1_{\Gamma_0}\beta(\mathbb{M} v)_{\tau}\\-[\mathcal{N}_E(\xi)n]_{\tau}+[\mathcal{N}_G(\xi)]_{\tau}+G(u,\xi+\eta^S) &  \text{on}\ (0,\infty)\times \partial\Omega,
\end{array}
\right. 
\end{equation*}
and the initial conditions
\begin{equation*}\label{nav4c*}
\left\{
\begin{array}{rl}
u(0,\cdot )=\widetilde{u}^0-w^S  &\text{in}\ \Omega,\\
\xi(0, \cdot)=\eta^0-\eta^S &\text{in}\  \omega,\\
\partial_t \xi(0,\cdot)=\eta^1&\text{in}\  \omega.\\
\end{array}
\right.
\end{equation*}
\end{proof}
\section{Feedback stabilization of the linear system}
\label{sl}
We consider the linear system associated to \eqref{nav3} and \eqref{nav3b} 
\begin{equation}
\label{oseenl}
\left\{
\begin{array}{rl}
\partial_t \widetilde{w}+(w^S\cdot \nabla)\widetilde{w}+(\widetilde{w}\cdot \nabla)w^S-\nabla\cdot\mathbb{T}(\widetilde{w},\widetilde{q})-\nabla\cdot\mathcal{L}^1(\xi)+\mathcal{L}^2(\xi,\partial_t\xi)=\widetilde{f}& \text{in} \ (0,\infty)\times \Omega,\\
\nabla \cdot \widetilde{w}=0& \text{in} \ (0,\infty)\times \Omega,\\
\partial_{tt} \xi+A_1\xi+A_2 \partial_t \xi+\mathcal{T}^*(\mathcal{L}^1(\xi)n)=-\mathcal{T}^*(\mathbb{T}(\widetilde{w},\widetilde{q})n)+\widetilde{h} & \text{in} \ (0,\infty)\times \omega,
\end{array}
\right.
\end{equation}
with the boundary conditions
\begin{equation}
\label{oseenbl}
\left\{
\begin{array}{rl}
[\widetilde{w}-\mathcal{T}\partial_t \xi]_{n}=1_{\Gamma_0}(\mathbb{M}v)_{n} &  \text{on} \ (0,\infty)\times\partial\Omega,\\
\left[ 2\nu D(\widetilde{w}) n + \beta \left(\widetilde{w}-\mathcal{T}\partial_t \xi\right)+\mathcal{L}^1(\xi)n+\mathcal{L}^3(\xi)\right]_{\tau}=1_{\Gamma_0}\beta(\mathbb{M} v)_{\tau}+\widetilde{g} &\text{on} \ (0,\infty)\times\partial\Omega,
\end{array}
\right.
\end{equation}
and the initial conditions
\begin{equation}\label{nav5c}
\left\{
\begin{array}{rl}
\widetilde{w}(0,\cdot )=w^0=\widetilde{u}^0-w^S  &\text{in}\ \Omega,\\
\xi(0, \cdot)=\xi^0=\eta^0-\eta^S &\text{in}\  \omega,\\
\partial_t \xi(0,\cdot)=\xi^1=\eta^1&\text{in}\  \omega,\\
\end{array}
\right.
\end{equation}
where
\begin{multline}
\label{s11}
\mathcal{L}^1(\xi)=Q^1\xi+Q^2\partial_s\xi+Q^3\partial^2_{ss}\xi,\\ \mathcal{L}^2(\xi,\partial_t\xi)= Q^4\xi+Q^5\partial_s\xi+Q^6\partial^2_{ss}\xi+Q^7\partial_t\xi+Q^8\partial^2_{st}\xi,\\	\mathcal{L}^3(\xi)=Q^9\xi+Q^{10}\partial_s\xi.
\end{multline}
We set 
\begin{equation}
\label{s6}
\mathcal{L}^{2,1}(\xi)=Q^4\xi+Q^5\partial_s\xi+Q^6\partial_{ss}^2\xi,\quad \mathcal{L}^{2,2}(\partial_t\xi)=Q^7\partial_t\xi+Q^8\partial^2_{ts}\xi.
\end{equation}
Here, we recall that $Q^i$, $i=1,...,8$ are some operators that depend on the stationary state, of regularity $W^{1,\infty}(\Omega)$ and $Q^9,\ Q^{10}$ are of regularity $W^{1/2,\infty}(\partial\Omega)$.

The next step consists to reformulate the system \eqref{oseenl}, \eqref{oseenbl} to get an evolution problem. To do so, we need to lift $\widetilde{g}$ in the boundary conditions. Then, we set the following lemma that concerns the instationary Stokes problem with Navier boundary conditions with a non negative friction.  
\begin{Lemma}
\label{lem}
Let $\beta \geq 0$ and $\gamma_0>0$. Let $(\widetilde{v},\widetilde{\pi})$ that verifies the system
\begin{equation}
\label{sys3*}
\left\{
\begin{array}{rl}
\partial_t \widetilde{v}+\gamma_0\widetilde{v}-\nabla \cdot \mathbb{T}(\widetilde{v},\widetilde{\pi})=0 &\quad \text{in }(0,\infty) \times \Omega,\\
\nabla \cdot \widetilde{v}=0 &\quad \text{in }(0,\infty) \times \Omega,\\
\widetilde{v}_{n}=0 & \quad \text{on }(0,\infty) \times \partial\Omega,\\
\left[2\nu D(\widetilde{v})n+\beta \widetilde{v}\right]_{\tau}=\widetilde{g} &\quad\text{on }(0,\infty) \times \partial\Omega,\\
\widetilde{v}(0,\cdot)=0 &\quad \text{in }\Omega.
\end{array}
\right.
\end{equation}
If
\begin{equation}
\label{LM*}
\widetilde{g}\in W_{\gamma}^{1/4}(0,\infty;[H^{1/2}(\partial\Omega)]^3,[L^2(\partial\Omega)]^3),\quad \widetilde{g}_{n_0}=0,
\end{equation}
for $\gamma\in [0,\gamma_0[$, then the problem \eqref{sys3*} admits a unique solution that satisfies the estimate
\begin{equation}
\label{estim*}
\left\| \widetilde{v}\right\|_{W_{\gamma}(0,\infty,[H^2(\Omega)]^3,[L^2(\Omega)]^3)}^2+\left\|  \nabla\widetilde{\pi}\right\|_{L^2_{\gamma}(0,\infty;[L^2(\Omega)]^3)}^2 
\\
\leq 
C
\left\| \widetilde{g}\right\|_{W_{\gamma}^{1/4}(0,\infty;[H^{1/2}(\partial\Omega)]^3,[L^2(\partial\Omega)]^3)}^2,
\end{equation}
where $C$ is a positive constant.
\end{Lemma}
\begin{Remark}
Since $\gamma_0>0$, then the system \eqref{sys3*} admits a unique solution for any $\beta\geq 0$ that is exponentially stable. If $\gamma_0=\beta=0$ the system is no more stable hence the importance of taking $\gamma_0>0$.
\end{Remark}
\begin{proof}
In the proof, we use the fact that the condition \eqref{LM*} is equivalent to
$$
\widetilde{g}\in W_{\gamma}^{1/2}(0,\infty;[H^{1}(\Omega)]^3,[L^2(\Omega)]^3),\quad \widetilde{g}_{n_0}=0.
$$ 
This fact is obtained in \cite[p.21, Theorem 2.3]{Lions}. 
This type of demonstration can be found in \cite{RS}, for the sake of completeness, we recall the main steps of the proof here and see how we can adapt some points of the proof in our case.
	
We consider the following problem
\begin{equation}\label{c1}
\left\{
\begin{array}{rlc}
\gamma_0z-\nabla\cdot \mathbb{T}(z,\chi)=0  & \quad \text{in }(0,\infty) \times \Omega,\\
\nabla \cdot z=0  &  \quad \text{in }(0,\infty) \times \Omega,\\
z_{n}=0 &\quad \text{on }(0,\infty) \times \partial\Omega,\\
\left[ 2\nu D(z)n +\beta z\right]_{\tau} =\widetilde{g} &\quad \text{on }(0,\infty) \times \partial \Omega.
\end{array}
\right.
\end{equation}
By a density argument, we assume that $\widetilde{g}\in C_0^{\infty}(\mathbb{R}_+,[H^1(\Omega)]^3)$ and $\widetilde{g}(t)=0$ if $t<0$. Then, the system \eqref{c1}, admits a unique solution $(z,\chi)\in C_0^{\infty}(\mathbb{R}_+,[H^2(\Omega)]^3)\times C_0^{\infty}(\mathbb{R}_+,H^1(\Omega)/\mathbb{R})$ (see \cite{BdV}). We set then, $\widetilde{z}=\widetilde{v}-z$, $\widetilde{\chi}=\widetilde{\pi}-\chi$. Then, $(\widetilde{z},\widetilde{\chi})$ satisfies the system
\begin{equation}\label{c2}
\left\{
\begin{array}{rlc}
\partial_t\widetilde{z}+\gamma_0\widetilde{z}-\nabla\cdot \mathbb{T}(\widetilde{z},\widetilde{\chi})=-\partial_tz  & \quad \text{in }(0,\infty) \times \Omega,\\
\nabla \cdot \widetilde{z}=0  &  \quad \text{in }(0,\infty) \times \Omega,\\
\widetilde{z}_{n}=0 &\quad \text{on }(0,\infty) \times \partial\Omega,\\
\left[ 2\nu D(\widetilde{z})n +\beta \widetilde{z}\right]_{\tau} =0 &\quad \text{on }(0,\infty) \times \partial \Omega,\\
\widetilde{z}(0,\cdot)=0 &\quad \text{in }\Omega.
\end{array}
\right.
\end{equation}
The system above admits a solution using the fact that the operator
$$
\mathbb{A}=\mathcal{P}\Delta-\gamma_0 \Id, \quad \mathcal{D}(\mathbb{A})=\{z\in [H^2(\Omega)]^3 \,;\,z\cdot n=0,\quad \left[ 2\nu D(z)n +\beta z\right]_{\tau} =0 \},
$$
generates an analytic semigroup exponentially stable where $\mathcal{P}$ designates the Leray projection. This is a direct consequence of \cite[Theorem 2.12, p.115]{BDDM}. We deduce then the existence of a unique solution $(\widetilde{v},\widetilde{\pi})\in C_0^{\infty}(\mathbb{R}_+,[H^2(\Omega)]^3)\times C_0^{\infty}(\mathbb{R}_+,H^1(\Omega)/ \mathbb{R})$ of the system \eqref{sys3*}. 
	
Now, we show $\eqref{estim*}$. Let $\phi \in C_0^\infty(\mathbb{R}_+,[L^2(\Omega)]^3)$ such that $\nabla \cdot\phi=0$, $\phi\cdot n=0$, and we suppose that $\phi=0$ on $(T,\infty)\times\Omega$ where $T>0$. We consider the system
\begin{equation}\label{c3}
\left\{
\begin{array}{rlc}
\partial_t\psi-(\gamma_0-\gamma)\psi+\nabla\cdot\mathbb{T}(\psi,\theta)=\phi  & \quad \text{in }\mathbb{R} \times \Omega,\\
\nabla \cdot \psi=0  &  \quad \text{in }\mathbb{R} \times \Omega,\\
\psi_{n}=0 &\quad \text{on }\mathbb{R} \times \partial\Omega,\\
\left[ 2\nu D(\psi)n+\beta \psi\right]_{\tau} =0 &\quad \text{on }\mathbb{R} \times \partial\Omega,\\
\psi(T)=0 &\quad\text{in }\Omega.\\
\end{array}
\right.
\end{equation}
We observe that $\psi=0$, $\theta=0$ on $(T,\infty)\times\Omega$. The system \eqref{c3} admits a unique solution
$$(\psi,\theta)\in W(\mathbb{R};[H^2(\Omega)]^3,[L^2(\Omega)]^3)\times L^2(\mathbb{R};H^1(\Omega)/\mathbb{R}),$$ 
exponentially stable satisfying
\begin{equation}
\label{h}
\left\|\psi\right\|_{W(\mathbb{R};[H^2(\Omega)]^3,[L^2(\Omega)]^3)}+\left\|\theta\right\|_{L^2(\mathbb{R};H^1(\Omega)/\mathbb{R})} \leq C\left\|\phi \right\|_{L^2(\mathbb{R};[L^2(\Omega)]^3)}.  
\end{equation}
We have
\begin{multline*}
\int_\mathbb{R}\int_{\Omega}e^{\gamma t} \widetilde{v}\cdot \phi \ dy dt=\int_\mathbb{R}\int_\Omega e^{\gamma t} \widetilde{v} \cdot (\partial_t\psi-(\gamma_0-\gamma)\psi+\nabla\cdot\mathbb{T}(\psi,\theta)) \ dy dt
\\=-\int_\mathbb{R} \int_\Omega e^{\gamma t}( \partial_t\widetilde{v}+\gamma_0\widetilde{v}-\nabla\cdot\mathbb{T}(\widetilde{v},\widetilde{\pi}))\cdot\psi \ dy dt-\int_\mathbb{R}\int_{\partial\Omega} e^{\gamma t}\widetilde{g}_\tau\cdot\psi_\tau \ d\Gamma dt\\=-\int_\mathbb{R}\int_{\partial\Omega} e^{\gamma t}\widetilde{g}_\tau\cdot\psi_\tau \ d\Gamma dt .
\end{multline*}
Using \eqref{h}, we get 
\begin{equation}
\label{h1}
\left\|e^{\gamma t}\widetilde{v}\right\|_{L^2(\mathbb{R};[L^2(\Omega)]^3)}\leq C \left\|e^{\gamma t}\widetilde{g}\right\|_{L^2(\mathbb{R};[H^1(\Omega)]^3)}.  
\end{equation}
In the other hand, we extend $n$ on $\Omega$ and we obtain
\begin{equation*}
\int_{\partial\Omega}e^{\gamma t}\partial_t\widetilde{g}\cdot\psi\ d\Gamma=\sum_{i}\int_\Omega e^{\gamma t}\partial_i((n)_i\partial_t \widetilde{g})\cdot\psi\ dy\\ + \sum_{i}\int_\Omega  e^{\gamma t}((n)_i\partial _t\widetilde{g})\cdot\partial_i\psi\ dy.
\end{equation*}
Then, 
\begin{equation}
\label{h4}
\int_\mathbb{R}\int_\Omega e^{\gamma t}\partial_t\widetilde{v} \cdot\phi \ dy dt= -\sum_{i}\int_\mathbb{R}\int_\Omega\left(  e^{\gamma t}\partial_i((n)_i\partial_t \widetilde{g})\cdot\psi+ e^{\gamma t}((n)_i\partial_t\widetilde{g})\cdot\partial_i\psi\right)\ dy dt.
\end{equation}
To estimate $t\longrightarrow e^{\gamma t}\partial_t\widetilde{v}(t)$ in $L^2(\mathbb{R},[L^2(\Omega)]^3)$, we notice that 
\begin{multline}
\int_{\mathbb{R}}\int_\Omega e^{\gamma t}\partial_i((n)_i\partial_t \widetilde{g})\cdot\psi \ dy dt=-\gamma\int_{\mathbb{R}}\int_\Omega e^{\gamma t}\partial_i((n)_i \widetilde{g})\cdot\psi \ dy dt-\int_{\mathbb{R}}\int_\Omega e^{\gamma t}\partial_i((n)_i \widetilde{g})\cdot\partial_t\psi \ dy dt .
\end{multline}
Then,
\begin{multline}
\label{h3}
\left| \left\langle  e^{\gamma t}\partial_i((n)_i\partial_t \widetilde{g}),\psi\right\rangle_{L^2(\mathbb{R};[L^2(\Omega)]^3)}\right| \leq C \left\|e^{\gamma t}\widetilde{g}\right\|_{L^2(\mathbb{R};[H^1(\Omega)]^3)}  \bigg(  \left\| \partial_t\psi\right\|_{L^2(\mathbb{R},[L^2(\Omega)]^3)}\\+\gamma\left\|\psi\right\|_{L^2(\mathbb{R},[L^2(\Omega)]^3)} \bigg).
\end{multline} 
Moreover, 
\begin{equation*}
\left\langle  e^{\gamma t}((n)_i\partial_t \widetilde{g}),\partial_i\psi\right\rangle_{L^2(\mathbb{R};[L^2(\Omega)]^3)}=\left\langle \partial_t(e^{\gamma t}\widetilde{g}),(n)_i\partial_i\psi\right\rangle_{L^2(\mathbb{R};[L^2(\Omega)]^3)}-\gamma \left\langle e^{\gamma t}\widetilde{g},(n)_i\partial_i\psi\right\rangle_{L^2(\mathbb{R};[L^2(\Omega)]^3)}.
\end{equation*}
Then, using \cite[Proposition 2.3, Proposition 2.8]{MR2384339}, we obtain
\begin{multline}
\label{h2}
\left|\left\langle  e^{\gamma t}((n)_i\partial_t \widetilde{g}),\partial_i\psi\right\rangle_{L^2(\mathbb{R};[L^2(\Omega)]^3)}\right|\leq C \left\|\widetilde{ g}\right\|_{W^{1/2}(\mathbb{R};[H^1(\Omega)]^3,[L^2(\Omega)]^3)}\left\| \nabla\psi\right\|_{W^{1/2}(\mathbb{R};[H^1(\Omega)]^9,[L^2(\Omega)]^9)}\\.\leq C \left\| \widetilde{g}\right\|_{W^{1/2}(\mathbb{R};[H^1(\Omega)]^3,[L^2(\Omega)]^3)} \left\|\psi\right\|_{W(\mathbb{R};[H^2(\Omega)]^3,[L^2(\Omega)]^3)}.
\end{multline}
Combining \eqref{h2}, \eqref{h3}, \eqref{h4} and \eqref{h} we have
\begin{equation}
\label{lol1}
\left\|e^{\gamma t}\partial_t\widetilde{v}\right\|_{L^2(\mathbb{R};[L^2(\Omega)]^3)}\leq C \left\|e^{\gamma t}\widetilde{g}\right\|_{W^{1/2}(\mathbb{R};[H^1(\Omega)]^3,[L^2(\Omega)]^3)}.
\end{equation}
Finally, from the classical elliptic estimates of the Stokes problem, we have
\begin{equation}
\label{lol}
\left\|\widetilde{v}(t)\right\|_{[H^2(\Omega)]^3}+ \left\|\nabla\widetilde{\pi}(t)\right\|_{[L^2(\Omega)]^3}\leq C\left( \left\|\partial_t\widetilde{v}(t)\right\|_{[L^2(\Omega)]^3}+\left\|\widetilde{g}(t)\right\|_{[H^1(\Omega)]^3}\right). 
\end{equation}
Using \eqref{lol} and \eqref{lol1}, we obtain \eqref{estim*}.
\end{proof}

We set 
$$
w=\widetilde{w}-\widetilde{v},\quad q=\widetilde{q}-\widetilde{\pi}.
$$
Then, from \cref{lem}, we only need to consider the linear system 
\begin{equation}
\label{oseen2}
\left\{
\begin{array}{rl}
\partial_t w+(w^S\cdot \nabla)w+(w\cdot \nabla)w^S-\nabla\cdot\mathbb{T}(w,q)-\nabla\cdot\mathcal{L}^1(\xi)+\mathcal{L}^2(\xi,\partial_t\xi)=f& \text{in} \ (0,\infty)\times \Omega,\\
\nabla \cdot w=0& \text{in} \ (0,\infty)\times \Omega,\\
\partial_{tt} \xi+A_1\xi+A_2 \partial_t \xi+\mathcal{T}^*(\mathcal{L}^1(\xi)n)=-\mathcal{T}^*(\mathbb{T}(w,q)n)+h & \text{in} \ (0,\infty)\times \omega,
\end{array}
\right.
\end{equation}
with the boundary conditions
\begin{equation}
\label{oseenb2}
\left\{
\begin{array}{rl}
[w-\mathcal{T}\partial_t \xi]_{n}=1_{\Gamma_0}(\mathbb{M}v)_{n} &  \text{on} \ (0,\infty)\times\partial\Omega,\\
\left[ 2\nu D(w) n + \beta \left(w-\mathcal{T}\partial_t \xi\right)+\mathcal{L}^1(\xi)n+\mathcal{L}^3(\xi)\right]_{\tau}=1_{\Gamma_0}\beta(\mathbb{M} v)_{\tau} &\text{on} \ (0,\infty)\times\partial\Omega,
\end{array}
\right.
\end{equation}
and the initial conditions
\begin{equation}\label{oseenc}
\left\{
\begin{array}{rl}
w(0,\cdot )=w^0  &\text{in}\ \Omega,\\
\xi(0, \cdot)=\xi^0 &\text{in}\  \omega,\\
\partial_t \xi(0,\cdot)=\xi^1&\text{in}\  \omega,\\
\end{array}
\right.
\end{equation}
where 
\begin{equation}
\label{fh}
f=\widetilde{f}-(w^S\cdot \nabla)\widetilde{v}-(\widetilde{v}\cdot \nabla)w^S+\gamma_0\widetilde{v},\quad h=\widetilde{h}-\mathcal{T}^*(\mathbb{T}(\widetilde{v},\widetilde{\pi})n).
\end{equation}
In this section, we prove that the system \eqref{oseen2}, \eqref{oseenb2} and \eqref{oseenc} is exponentially stabilisable.

We define
\begin{equation*}
\mathbb{H}=\{
\left( w,\eta_1,\eta_2\right)\in [L^2(\Omega)]^3 \times\mathcal{D}(A_1^{1/2})\times L^2_0(\omega)
\ ; \ \nabla \cdot w= 0 \ \text{in}  \ \Omega, 
\left[{w}-\mathcal{T}\eta_2\right]_{n}=0 \ \text{on}\ \partial \Omega
\} ,
\end{equation*}
with the inner product 
\begin{equation*}
\left\langle \begin{pmatrix}
w\\\eta_1\\ \eta_2
\end{pmatrix},\begin{pmatrix}
v\\ \xi_1\\ \xi_2
\end{pmatrix} \right\rangle_{\mathbb{H}} =\left\langle w, v \right\rangle_{[L^2(\Omega)]^3}+\left\langle A_1^{1/2} \eta_1, A_1^{1/2}\xi_1\right\rangle _{L^2(\omega)}
+\left\langle \eta_2,\xi_2\right\rangle _{L^2(\omega)}.
\end{equation*}
We define also
\begin{align}
\mathbb{V}&= \left([H^1(\Omega)]^3\times\mathcal{D}(A_1^{3/4})\times \mathcal{D}(A_1^{1/4})\right)\cap \mathbb{H},
\end{align}
and the orthogonal projection $\mathbb{P}$ 
$$
\mathbb{P}:[L^2(\Omega)]^3 \times\mathcal{D}(A_1^{1/2})\times L^2_0(\omega) \longrightarrow \mathbb{H}.
$$
Finally, we set the operator $\mathcal{A}_S$
\begin{equation}
\label{as}
\mathcal{A}_S \begin{pmatrix}
w\\ \eta_1\\ \eta_2
\end{pmatrix}= \begin{pmatrix}
\nu \Delta w-(w^S\cdot \nabla)w-(w\cdot \nabla)w^S+\nabla\cdot\mathcal{L}^1(\eta_1)-\mathcal{L}^2(\eta_1,\eta_2)\\ \eta_2\\ -A_1\eta_1-A_2\eta_2-\mathcal{T}^*(2\nu D(w)n+\mathcal{L}^1(\eta_1)n)
\end{pmatrix},
\end{equation}
\begin{multline}
\mathcal{D}(\mathcal{A}_S)
=\{\left( w,\eta_1,\eta_2\right)\in \left([H^2(\Omega)]^3\times \mathcal{D}(A_1)\times \mathcal{D}(A_1^{1/2})\right)\cap \mathbb{V},
\\
\left[2\nu D({w})n+\beta ({w}-\mathcal{T}\eta_2)+\mathcal{L}^1(\eta_1)n+\mathcal{L}^3(\eta_1)\right]_{\tau}=0 \quad \text{on} \ \partial \Omega 
\},
\end{multline}
and the operator $A_S$ defined as follows
\begin{equation}
\label{asd}
A_S=\mathbb{P}\mathcal{A}_S,\quad \mathcal{D}(A_S)=\mathcal{D}(\mathcal{A}_S).
\end{equation}
\begin{Proposition}
\label{acc}
There exists $\lambda_0>0$ such that the operator $\lambda_0\Id-A_S$ is the infinitesimal generator of a strongly continuous semigroup of contractions on $\mathbb{H}$.
\end{Proposition}
\begin{proof}
Let $W=(w,\xi_1,\xi_2)\in\mathcal{D}(A_S)$. We show that $\lambda_0\Id-A_S$ is dissipative. 	
\begin{multline}
\left\langle (\lambda_0\Id-A_S)W,W \right\rangle=\lambda_0\int_{\Omega} \left| w\right|^2 \ dy+\int_{\Omega} (\nabla w^S)^*w\cdot w \ dy+2\nu\int_\Omega\left|Dw\right|^2 \ dy\\+\int_{\Omega}\mathcal{L}^1(\xi_1):\nabla w \ dy+\int_{\Omega}\mathcal{L}^2(\xi_1,\xi_2)\cdot w \ dy+\int_{\partial\Omega}\beta\left|(w-\mathcal{T}\xi_2)_{\tau}\right|^2 \ d\Gamma\\+\int_{\partial\Omega}[\mathcal{L}^3(\xi_1)]_{\tau}\cdot(w-\mathcal{T}\xi_2)_{\tau}\ d\Gamma +\int_{\omega}\left| A_2^{1/2}\xi_2\right|^2 \ ds+\lambda_0\int_\omega \left| A_1^{1/2}\xi_1\right|^2 \ ds\\+\lambda_0\int_\omega \left| \xi_2\right| ^2\ ds.  
\end{multline}
From \eqref{s11} and the fact that $Q^i$, $i=1,...,8$ are of regularity $W^{1,\infty}(\Omega)$, then
\begin{equation}
\label{s3}
\left\| \mathcal{L}^1(\xi_1)\right\|_{[L^2(\Omega)]^9}^2 \leq C\left\| A_1^{1/2}\xi_1\right\|_{L^2_0(\omega)}^2,
\end{equation}
\begin{equation}
\label{s4}
\left\| \mathcal{L}^2(\xi_1,\xi_2)\right\|_{[L^2(\Omega)]^3}^2 \leq C\left( \left\| A_1^{1/2}\xi_1\right\|_{L^2_0(\omega)}^2+\left\| A_2^{1/2}\xi_2\right\|_{L^2_0(\omega)}^2\right).
\end{equation}
Since $\xi_1$ is a periodic function, we get
$$
\left\|\xi_1\right\|_{L^6(\omega)}+\left\|\partial_s\xi_1\right\|_{L^6(\omega)}\leq C\left\|\xi_1\right\|_{H^2(\omega)}\leq C\left\|A_1^{1/2}\xi_1\right\|_{L_0^2(\omega)}.
$$
Using that $Q^9$ and $Q^{10}$ are of regularity $W^{1/2,\infty}(\partial\Omega)\hookrightarrow L^3(\partial\Omega)$, then we obtain 
\begin{equation}
\label{s5}
\left\| \mathcal{L}^3(\xi_1)\right\|_{[L^2(\Omega)]^3}^2\leq C \left\| A_1^{1/2}\xi_1\right\|_{L^2_0(\omega)}^2.
\end{equation}
Using \eqref{s3}, \eqref{s4} and \eqref{s5}, we find
\begin{multline}
\int_{\Omega} (\nabla w^S)^*w\cdot w \ dy+\int_{\Omega}\mathcal{L}^1(\xi_1):\nabla w \ dy+\int_{\Omega}\mathcal{L}^2(\xi_1,\xi_2)\cdot w \ dy\\+\int_{\partial\Omega}[\mathcal{L}^3(\xi_1)]_{\tau}\cdot(w-\mathcal{T}\xi_2)_{\tau}\ d\Gamma\geq -\varepsilon \bigg(\left\| w\right\|^2_{[H^1(\Omega)]^3} + \left\| A_2^{1/2}\xi_2\right\|_{L^2(\Omega)}^2\bigg)\\- C\left( \left\|w \right\|^2_{[L^2(\Omega)]^3}+\left\| \xi_2\right\|^2_{L^2(\omega)} +\left\| A_1^{1/2}\xi_1\right\|^2_{L^2(\omega)}\right).
\end{multline}
Thanks to the classical Korn's inequality, we obtain for $\lambda_0$ large enough
$$
\left\langle (\lambda_0\Id-A_S)W,W \right\rangle \geq 0.
$$
It remains to show that the oprator $\lambda_0\Id-A_S$ is m-dissipative: we prove that $\lambda_0\Id-A_S$ is onto.

Let $F=\begin{pmatrix}
f\\ g\\h
\end{pmatrix}\in \mathbb{H}$ and we consider the equation
\begin{equation}
\label{onto*}
(\lambda_0\Id-A_S)\begin{pmatrix}
w\\ \eta_1\\ \eta_2
\end{pmatrix}=F,
\end{equation}

\begin{subequations}\label{sys*}
\begin{align}
\lambda_0 w -\nabla \cdot \mathbb{T}(w,q)+(w^S\cdot \nabla)w+(w\cdot \nabla)w^S-\nabla\cdot\mathcal{L}^1(\eta_1)+\mathcal{L}^2(\eta_1,\eta_2)=f \quad \text{in}\ \Omega, \label{sys01*}\\
\nabla \cdot w=0 \quad \text{in}\  \Omega,\label{sys02*}\\
\lambda_0 \eta_1-\eta_2 =g \quad \text{in}\ \omega,\label{sys03*}\\
\lambda_0\eta_2+A_1\eta_1+A_2 \eta_2=-\mathcal{T}^*(\mathbb{T}(w,q)n+\mathcal{L}^1(\eta_1)n)+h \quad \text{in}\  \omega,\label{sys04*}\\
\relax [w-\mathcal{T}\eta_2]_{n}=0 \quad \text{on}\ \partial\Omega,\label{sys05*}\\
\left[2\nu D(w)n+\beta (w-\mathcal{T}\eta_2)+\mathcal{L}^1(\eta_1)n+\mathcal{L}^3(\eta_1)\right]_{\tau}=0\quad \text{on}\ \partial\Omega.\label{sys06*}
\end{align}
\end{subequations}
To solve the system above, we introduce the space
$$
\mathcal{V}=
\left\lbrace
(\phi,\xi)\in [H^1(\Omega)]^3\times \mathcal{D}(A_1^{1/2}) \mid 
\nabla \cdot \phi= 0\quad \text{in}  \ \Omega, 
\quad 
[\phi-\mathcal{T}\xi]_{n}=0 \quad\text{on}\ \partial\Omega
\right\rbrace.
$$
Then, solving the equation \eqref{onto*} is reduced to solve the following variational problem: one need to find $(w,\eta_2)\in \mathcal{V}$ such that
\begin{equation}\label{fv*}
a\left( \begin{pmatrix}w\\ \eta_2\end{pmatrix},
\begin{pmatrix}\phi \\ \xi\end{pmatrix}\right) 
= 
L\begin{pmatrix}\phi \\ \xi\end{pmatrix}
\quad \left(\begin{pmatrix}\phi \\ \xi\end{pmatrix}\in \mathcal{V}\right),
\end{equation}
with $a: \mathcal{V} \times \mathcal{V}\longrightarrow \mathbb{R}$ the bilinear form given by
\begin{multline*}
a\left( \begin{pmatrix}w\\ \eta_2\end{pmatrix},
\begin{pmatrix}\phi \\ \xi\end{pmatrix}\right)
=\lambda_0 \int_{\Omega} w\cdot \phi \ dy+\int_{\Omega}((w^S\cdot \nabla)w+(w\cdot \nabla)w^S)\cdot \phi \ dy+2\nu \int_{\Omega}D(w):D(\phi)\ dy\\ +\frac{1}{\lambda_0}\int_{\Omega}\mathcal{L}^1(\eta_2):\nabla \phi \ dy+\frac{1}{\lambda_0}\int_{\Omega}\mathcal{L}^{2,1}(\eta_2)\cdot \phi \ dy+\int_{\Omega}\mathcal{L}^{2,2}(\eta_2)\cdot \phi \ dy+\lambda_0 \int_{\omega} \eta_2\cdot \xi\  ds
\\+\int_{\omega}  (A_2\eta_2)\cdot \xi \ ds
+\frac{1}{\lambda_0}\int_{\omega}  (A_1^{1/2}\eta_2)\cdot (A_1^{1/2} \xi)  \ ds
+\int_{\partial \Omega} \beta [w-\mathcal{T}(\eta_2)]_{\tau}\cdot [\phi-\mathcal{T}(\xi)]_{\tau}  \ d\Gamma\\
+\frac{1}{\lambda_0}\int_{\partial \Omega} [\mathcal{L}^3(\eta_2)]_{\tau}\cdot [\phi-\mathcal{T}(\xi)]_{\tau}  \ d\Gamma,
\end{multline*}
and $L:\mathcal{V}\longrightarrow \mathbb{R}$ the linear form defined by
\begin{multline*}
L\begin{pmatrix}\phi \\ \xi\end{pmatrix}
= \int_{\Omega} f \cdot \phi \ dy+\int_{\omega} h\cdot \xi\ ds-\frac{1}{\lambda_0}\int_{\omega} (A_1^{1/2}g)\cdot (A_1^{1/2}\xi) \ ds\\-\frac{1}{\lambda_0}\int_{\Omega}\mathcal{L}^1(g):\nabla \phi \ dy-\frac{1}{\lambda_0}\int_{\Omega}\mathcal{L}^{2,1}(g)\cdot \phi \ dy-\frac{1}{\lambda_0}\int_{\partial \Omega} [\mathcal{L}^3(g)]_{\tau}\cdot [\phi-\mathcal{T}(\xi)]_{\tau}  \ d\Gamma.
\end{multline*}
The bilinear form $a$ is continuous and coercive on $\mathcal{V}$ thanks to the classical Korn's inequality and $L$ is continuous on $\mathcal{V}$. Using the Lax-Milgram theorem, there exists a unique $(w,\eta_2)\in \mathcal{V}$ that is solution to \eqref{fv*}.	
Now, taking $\xi=0$ and $\phi\in \mathcal{D}_\sigma(\Omega)$, the equation \eqref{fv*} becomes
\begin{multline*}
\lambda_0 \int_{\Omega} w\cdot \phi \ dy + \int_{\Omega}((w^S\cdot \nabla)w+(w\cdot \nabla)w^S)\cdot \phi \ dy +2\nu \int_{\Omega}D(w):D(\phi)\ dy
\\+\int_{\Omega}\mathcal{L}^1(\eta_1):\nabla \phi \ dy+\int_{\Omega}\mathcal{L}^{2}(\eta_1,\eta_2)\cdot \phi \ dy=\int_{\Omega} f \cdot \phi \ dy,
\end{multline*}
that is equivalent to
\begin{equation*}
\left\langle \lambda_0 w+(w^S\cdot \nabla)w+(w\cdot \nabla)w^S-\nu\Delta w-\nabla \cdot\mathcal{L}^1(\eta_1)+\mathcal{L}^2(\eta_1,\eta_2)-f,\phi\right\rangle 
=0,\quad \forall \phi \in \mathcal{D}_\sigma(\Omega).
\end{equation*}
Using the De Rham theorem \cite[Proposition 1.2, p.14]{T}, 
we deduce the existence of a unique element $q\in L^2(\Omega)/\mathbb{R}$ such that \eqref{sys01*} is satisfied. 
In particular, we have  $\nabla\cdot\mathbb{T}(w,q)\in [L^2(\Omega)]^3$ and  $\mathbb{T}(w,q)\in [L^2(\Omega)]^9$. Consequently, we get $\mathbb{T}(w,q)n\in [H^{-1/2}(\partial \Omega)]^3$
and 
\begin{multline}
\label{im**}
\int_\Omega \mathbb{T}(w,q):D(\phi) \ dy-\left\langle\mathbb{T}(w,q)n,\phi\right\rangle_{H^{-1/2},H^{1/2}} 
\\=\int_\Omega (f-\lambda_0 w-(w^S\cdot \nabla)w-(w\cdot \nabla)w^S-\mathcal{L}^2(\eta_1,\eta_2))\cdot \phi \ dy-\int_\Omega\mathcal{L}^1(\eta_1):\nabla\phi \ dy\\+\int_{\partial\Omega}[\mathcal{L}^1(\eta_1)n]_{\tau}\cdot \phi_{\tau} \ d\Gamma,
\end{multline}	
for all $\phi\in [H^1(\Omega)]^3,\;\phi_{n}=0$. 
Taking $\xi=0$ in \eqref{fv*}, we obtain
\begin{multline}
\label{im2*}
\lambda_0 \int_{\Omega} w\cdot \phi \ dy+2\nu \int_{\Omega}D(w):D(\phi)\ dy
+\int_{\Omega}\mathcal{L}^1(\eta_1):\nabla \phi \ dy+\int_{\Omega}\mathcal{L}^{2}(\eta_1,\eta_2)\cdot \phi \ dy\\
+\int_{\Omega}((w^S\cdot \nabla)w+(w\cdot \nabla)w^S)\cdot \phi \ dy+\left\langle  [ \beta(w-\mathcal{T}(\eta_2))+\mathcal{L}^3(\eta_1)]_{\tau}, \phi  \ \right\rangle_{H^{-1/2},H^{1/2}}\\=\int_{\Omega} f \cdot \phi \ dy,
\end{multline}
for any $\phi\in [H^1(\Omega)]^3,\;\nabla\cdot \phi=0,\;\phi_{n}=0$. Comparing \eqref{im**} and \eqref{im2*} and taking into account the fact 
$$
\int_\Omega \mathbb{T}(w,q):D(\phi)\ dy=2\nu\int_\Omega D(w):D(\phi)\ dy, \quad \forall\phi\in [H^1(\Omega)]^3,\;\nabla\cdot \phi=0,\;\phi_{n}=0,
$$
we obtain for all $\phi\in [H^1(\Omega)]^3$ satisfying $\nabla\cdot \phi=0$ and $\phi_{n}=0$
\begin{equation}
\label{im1*}
-\left\langle \mathbb{T}(w,q)n,\phi\right\rangle_{H^{-1/2},H^{1/2}} =\left\langle [ \beta(w-\mathcal{T}\eta_2)+\mathcal{L}^1(\eta_1)n+\mathcal{L}^3(\eta_1)]_{\tau},\phi\right\rangle_{H^{-1/2},H^{1/2}}.
\end{equation}
Then, \eqref{im1*} is also satisfied for all $\phi\in [H^{1}(\Omega)]^3$, $\phi_{n}=0$ since we can always construct a divergence free function on $\Omega$ that coincides with $\phi$ at the boundary.
Plugging \eqref{im1*} in \eqref{im**}, we obtain
\begin{multline}
\label{im6*}
\int_{\Omega}((w^S\cdot \nabla)w+(w\cdot \nabla)w^S)\cdot \phi \ dy+2\nu\int_{\Omega} D(w):D(\phi)\ dy-\int_{\Omega}q\nabla\cdot\phi  \ dy\\+\left\langle[\beta(w-\mathcal{T}\eta_2)+\mathcal{L}^3(\eta_1)]_{\tau},\phi_{\tau} \right\rangle_{H^{-1/2},H^{1/2}}= \int_\Omega (f-\lambda_0 w-\mathcal{L}^2(\eta_1,\eta_2))\cdot \phi \ dy\\-\int_\Omega\mathcal{L}^1(\eta_1):\nabla\phi \ dy,
\end{multline}
for any $\phi\in [H^1(\Omega)]^3$, $\phi_{n}=0$.
	
Then, we deduce that $(w,q)$ is a weak solution of \eqref{sys01*}, \eqref{sys02*}, \eqref{sys05*} and \eqref{sys06*} in the sens of the \cite[Definition, p.10]{BdV}.
Since $\eta_2\in H^2(\omega)$ then $\mathcal{T}\eta_2\in [H^{2}(\partial\Omega)]^3$, we can apply \cite[Théorème 1.2]{BdV} to get that $(w,q) \in [H^2(\Omega)]^3\times H^1(\Omega)/\mathbb{R}$. 
	
Then, going back to \eqref{fv*}, we get
	
\begin{multline}
\int_{\omega}  (A_1^{1/2}\eta_1)\cdot (A_1^{1/2} \xi)  \ ds
=-\lambda \int_{\omega} \eta_2\cdot \xi\  ds
-\int_{\omega}  (A_2\eta_2)\cdot \xi \ ds
\\-\int_{\omega}  \mathcal{T}^*(\mathbb{T}(u,q)n+\mathcal{L}^1(\eta_1)n) \cdot \xi \ ds
+\int_{\omega} h\cdot \xi\ ds,
\end{multline}
for all $\xi\in\mathcal{D}(A_1^{1/2})$ where $\eta_1=\frac{1}{\lambda}(g+\eta_2)$. We notice that we have $\eta_1\in \mathcal{D}(A^{1/2})$, then $\mathcal{T}^*(\mathcal{L}^1(\eta_1)n)\in L^2_0(\omega)$. In the other hand,
we have $\mathbb{T}(w,q)n \in [H^{1/2}(\partial \Omega)]^{3}$, then it implies that $\mathcal{T}^*(\mathbb{T}(w,q)n) \in L^2_0(\omega)$. Moreover, 
since $\eta_2\in H^2(\omega)$, we deduce that $\eta_2\in \mathcal{D}(A_2)$. Thus, using the fact that $\mathcal{D}(A_1^{1/2})$ is dense in  $L^2_0(\omega)$, we obtain $A_1\eta_1 \in L^2_0(\omega)$.
	
Applying the Lumer-Phillips theorem, we deduce that $\lambda_0\Id-A_S$ generates a strongly continuous semigroup of contractions on $\mathbb{H}$.
\end{proof}

\begin{Proposition}
The adjoint operator of $A_S$ is given by
\begin{equation}
\label{as*}
A^*_S \begin{pmatrix}
\varphi\\ \zeta_1\\ \zeta_2
\end{pmatrix}=\mathbb{P} \begin{pmatrix}
\nu \Delta \varphi+(w^S\cdot \nabla)\varphi-(\nabla w^S)^*\varphi\\ -\zeta_2-A_1^{-1}\left( (\mathcal{L}^1)^*)(\nabla \varphi)+(\mathcal{L}^{2,1})^*\varphi+(\mathcal{L}^3)^*\varphi-((\mathcal{L}^3)^*(\mathcal{T}\zeta_2)\right) \\ A_1\zeta_1-A_2\zeta_2-\mathcal{T}^*(2\nu D(\varphi)n)-(\mathcal{L}^{2,2})^*\varphi
\end{pmatrix},
\end{equation}
and
\begin{multline}
\mathcal{D}(A^*_S)
=\{\left( \varphi,\zeta_1,\zeta_2\right)\in \left([H^2(\Omega)]^3\times \mathcal{D}(A_1)\times \mathcal{D}(A_1^{1/2})\right)\cap \mathbb{V},
\\
\left[2\nu D({\varphi})n+\beta ({\varphi}-\mathcal{T}\zeta_2)\right]_{\tau}=0 \quad \text{on} \ \partial \Omega 
\}.
\end{multline}
\end{Proposition}
\begin{proof}
Let $(w,\eta_1,\eta_2)\in\mathcal{D}(A_S)$ and $(\varphi,\zeta_1,\zeta_2)\in\mathcal{D}(A^*_S)$. We have
\begin{multline}
\left\langle A_S(w,\eta_1,\eta_2),(\varphi,\zeta_1,\zeta_2)\right\rangle=-2\nu\int_\Omega D(w):D(\varphi) \ dy + \int_\Omega (w^S\cdot\nabla)\varphi\cdot w\\-\int_\Omega(\nabla w^S)^*\varphi\cdot w \ dy-\int_\Omega \mathcal{L}^1(\eta_1):\nabla w \ dy-\int_\Omega \mathcal{L}^2(\eta_1,\eta_2)\cdot w \ dy\\-\int_{\partial\Omega}\beta(w-\mathcal{T}\eta_2)_{\tau}\cdot(\varphi-\mathcal{T}\zeta_2)_{\tau} \ d\Gamma-\int_{\partial\Omega}[\mathcal{L}^3(\eta_1)]_{\tau}\cdot(\varphi-\mathcal{T}\zeta_2)_{\tau} \ d\Gamma\\-\int_\omega (A_2\zeta_2)\cdot\eta_2 \ ds.
\end{multline}
Then
\begin{multline}
\left\langle A_S(w,\eta_1,\eta_2),(\varphi,\zeta_1,\zeta_2)\right\rangle=\nu\int_\Omega \Delta\varphi\cdot w \ dy + \int_\Omega (w^S\cdot\nabla)\varphi\cdot w\\-\int_\Omega(\nabla w^S)^*\varphi\cdot w \ dy-\int_\Omega \mathcal{L}^1(\eta_1):\nabla \varphi \ dy-\int_\Omega \mathcal{L}^2(\eta_1,\eta_2)\cdot \varphi \ dy\\-\int_{\partial\Omega}[\mathcal{L}^3(\eta_1)]_{\tau}\cdot(\varphi-\mathcal{T}\zeta_2)_{\tau} \ d\Gamma-\int_\omega (A_2\zeta_2)\cdot\eta_2 \ ds-\int_\omega\mathcal{T}^*(2\nu D(\varphi)n)\eta_2 \ ds.
\end{multline}
Using \eqref{s11} and \eqref{s6}, we obtain \eqref{as*}.
\end{proof}
\begin{Proposition}
For $\theta\in [0,1]$, we have 
\begin{equation}
\mathcal{D}((\lambda_0\Id-A_S)^{\theta})=[\mathcal{D}(A),\mathbb{H}]_{1-\theta},\quad \mathcal{D}((\lambda_0\Id-A_S^*)^{\theta})=[\mathcal{D}(A^*),\mathbb{H}]_{1-\theta},
\end{equation}
for $\lambda_0>0$ large enough.
\end{Proposition}
\begin{proof}
The proof is a direct consequence of \cref{acc} and of \cite[Proposition 6.1, p.171]{BDDM}.
\end{proof}
\begin{Proposition}
The operator $A_S$ defined by \eqref{as} and \eqref{asd} is the infinitesimal generator of an analytical semigroup on $\mathbb{H}$.
\end{Proposition}
\begin{proof}
We notice that $A^*_S=-A^*+O_S$ where $A$ is the operator defined by 
\begin{multline}
\mathcal{D}(\mathcal{A})
=\{\left( w,\eta_1,\eta_2\right)\in \left([H^2(\Omega)]^3\times \mathcal{D}(A_1)\times \mathcal{D}(A_1^{1/2})\right)\cap \mathbb{V},
\\
\left[2\nu D({w})n+\beta ({w}-\mathcal{T}\eta_2)\right]_{\tau}=0 \quad \text{on} \ \partial \Omega 
\},
\end{multline}
\begin{equation}
\mathcal{A} \begin{pmatrix}
w\\ \eta_1\\ \eta_2
\end{pmatrix}= \begin{pmatrix}
-\nu \Delta w\\ -\eta_2\\ A_1\eta_1+A_2\eta_2+\mathcal{T}^*(2\nu D(w)n)
\end{pmatrix},
\end{equation}
\begin{equation}
\mathcal{D}(A)=\mathcal{D}(\mathcal{A}), \quad A=\mathbb{P}\mathcal{A},
\end{equation}
and
\begin{equation*}
O_S \begin{pmatrix}
w\\ \xi_1\\ \xi_2
\end{pmatrix}=\mathbb{P}\begin{pmatrix}
(w^S\cdot \nabla)w-(\nabla w^S)^*w\\ -A_1^{-1}\left( (\mathcal{L}^1)^*)(\nabla w)+(\mathcal{L}^{2,1})^*w+(\mathcal{L}^3)^*w-((\mathcal{L}^3)^*(\mathcal{T}\xi_2)\right) \\ -(\mathcal{L}^{2,2})^*w
\end{pmatrix}.
\end{equation*}
We precise here
$$
\mathcal{D}((A^*)^{1/2})=\mathbb{V}.
$$
From \cite{MR3962841}, $A$ is the infinitesimal generator of an analytical semigroup on $\mathbb{H}$. Then, from \cite[Corollaire 2.4, p.81]{PZ}, we only need to show that $\mathbb{V}\subset \mathcal{D}(O_S)$. In fact, let $(w,\xi_1,\xi_2)\in\mathbb{V}$ and $(\phi,\eta_1,\eta_2)\in\mathbb{H}$.
First, we use \eqref{s6} and we observe that for $\eta_2\in H^1(\omega)\cap L^2_0(\omega)$ 
\begin{multline}
-\int_{\omega}(\mathcal{L}^{2,2})^*w\cdot \eta_2 \ ds=-\int_{\Omega}w\cdot\mathcal{L}^{2,2}(\eta_2) \ dy=-\int_{\Omega}w\cdot (Q^7\eta_2+Q^8\nabla\eta_2) \ dy\\
=-\int_{\Omega}w\cdot Q^7\eta_2 \ dy-\int_{\Omega}w\cdot Q^8\nabla\eta_2 \ dy.
\end{multline}
By integration by parts, we find that
$$
\left| \int_{\Omega}w\cdot Q^7\eta_2 \ dy+\int_{\Omega}w\cdot Q^8\nabla\eta_2 \ dy\right|\leq C\left\|w\right\|_{[H^1(\Omega)]^3}\left\|\eta_2\right\|_{L^2(\omega)}. 
$$
Then, by density, we get
\begin{multline}
\left\langle O_S\begin{pmatrix}
w\\ \xi_1\\ \xi_2
\end{pmatrix},\begin{pmatrix}
\varphi\\ \eta_1\\ \eta_2
\end{pmatrix} \right\rangle _{\mathbb{H}}\leq C\bigg( \left\|w\right\|_{[H^1(\Omega)]^3} \left( \left\|\varphi\right\|_{[L^2(\Omega)]^3}+\left\|A^{1/2}\eta_1\right\|_{L^2(\omega)} \right)\\+\left\|A^{1/2}\eta_1\right\|_{L^2(\omega)}\left\|\xi_2\right\|_{L^2(\omega)}+ \left\|w\right\|_{[H^1(\Omega)]^3}\left\|\eta_2\right\|_{L^2(\omega)}\bigg) \\\leq C \left\| (w,\xi_1,\xi_2)\right\|_{\mathbb{V}}\left\| (\varphi,\eta_1,\eta_2)\right\|_{\mathbb{H}}.  
\end{multline}
 We deduce that $A_S$ is the infinitesimal generator of an analytical semigroup on $\mathbb{H}$. Moreover, $A_S$ admits a compact resolvent.
\end{proof}

Let
$$
\mathcal{V}_n^{s}(\Gamma_0)=\{v\in H^s(\Gamma_0)\ ; \ \int_{\Gamma_0}v\cdot n \ d\Gamma=0 \},\quad s\geq 0.
$$
From \cite{RaymondStokes}, the system  \eqref{oseen2}, \eqref{oseenb2}, \eqref{oseenc} is equivalent to
\begin{equation}
\label{dyn}
\begin{array}{cccc}
\mathbb{P}W'&=& A_S\mathbb{P}W+\mathbb{P}F+Bv, & \mathbb{P}W(0)=\mathbb{P}W^0,\\
(\Id-\mathbb{P})W&=&(\Id-\mathbb{P})D(\mathbb{M}v),
\end{array}
\end{equation}
with
$$\xi_1=\xi,\quad \xi_2=\partial_t\xi,\quad
W=\begin{pmatrix}
w\\ \xi_1\\ \xi_2
\end{pmatrix},\quad F=\begin{pmatrix}
f\\ 0\\ h
\end{pmatrix},\quad
W^0=\begin{pmatrix}
w^0\\ \xi^0\\ \xi^1
\end{pmatrix},
$$
where $D\in\mathcal{L}(\mathcal{V}_n^{0}(\Gamma_0),[L^2(\Omega)]^3\times\mathcal{D}(A_1^{1/2})\times L^2_0(\omega))$ such that $Dv=(\overline{w},\eta_1,\eta_2)$ verifies the system 
\begin{equation*}
\left\{
\begin{array}{rl}
\lambda_0 \overline{w}-\nabla\cdot(\mathbb{T}(\overline{w},\pi)) +(w^S\cdot \nabla)\overline{w}+(\overline{w}\cdot\nabla) w^S-\nabla\cdot\mathcal{L}^1(\eta_1)+\mathcal{L}^2(\eta_1,\eta_2)=0& \text{in}\ \Omega,\\
\nabla \cdot \overline{w} =0& \text{in}\ \Omega,\\
\lambda_0\eta_1-\eta_2=0& \text{in}\ \omega,\\
\lambda_0\eta_2+A_1\eta_1+A_2\eta_2+\mathcal{T}^*(\mathbb{T}(\overline{w},\pi)n+\mathcal{L}^1(\eta_1)n)=0& \text{in}\ \omega,\\
\relax[\overline{w}-\mathcal{T}\eta_2\relax]_{n}=1_{\Gamma_0}v_n & \text{on} \ \partial\Omega,\\
\left[ 2\nu D(\overline{w}) n + \beta (\overline{w}-\mathcal{T}\eta_2)+\mathcal{L}^1(\eta_1)n+\mathcal{L}^3(\eta_1)\right]_{\tau}=1_{\Gamma_0}(\beta v_{\tau}) &\text{on} \ \partial\Omega.
\end{array}
\right.
\end{equation*}
with $\lambda_0 \in \rho(A_S)$, and the operator $B$ is defined by
$$
B:v\in\mathbb{U}=[H^{3/2}(\Gamma_0)]^3\longrightarrow (\lambda_0\Id-A_S)\mathbb{P}D(\mathbb{M}v)\in (\mathcal{D}(A_S^*))'. 
$$
We notice that 
$$
(\lambda_0\Id-A_S)^{-1+\varepsilon}B \in \mathcal{L}([H^{3/2}(\Gamma_0)]^3,\mathbb{H}),\quad \forall\varepsilon \in (0,1/4).
$$
The adjoint operator $D^*\in\mathcal{L}([L^2(\Omega)]^3\times\mathcal{D}(A_1^{1/2})\times L^2_0(\omega),[L^2(\Gamma_0)]^3)$ is defined for all $(\overline{f},\overline{g},\overline{h})\in [L^2(\Omega)]^3\times\mathcal{D}(A_1^{1/2})\times L^2_0(\omega)$ by
$$
\left\{
\begin{array}{c}
D^*(\overline{f},\overline{g},\overline{h})=\left. \left(  (\mathbb{T}(\phi,r)n)_{n}+(\mathbb{T}(\phi,r)n)_{\tau}\right)\right| _{\Gamma_0} ,\quad \text{  if } \beta_1>0,\\
D^*(\overline{f},\overline{g},\overline{h})=\left. (\mathbb{T}(\phi,r)n)_{n}\right| _{\Gamma_0} ,\quad \text{  if } \beta_1=0,
\end{array}
\right.
$$ 
such that $(\phi,r,\zeta_1,\zeta_2)\in [H^2(\Omega)]^3\times H^1(\Omega)/\mathbb{R}\times\mathcal{D}(A_1)\times\mathcal{D}(A_1^{1/2})$ is solution of the system
\begin{equation*}
\left\{
\begin{array}{rl}
\lambda_0 \phi-\nabla\cdot(\mathbb{T}(\phi,r))-(w^S\cdot \nabla)\phi+(\nabla w^S)^*\phi=\overline{f}& \text{in}\ \Omega,\\
\nabla \cdot \phi =0& \text{in}\ \Omega,\\
\lambda_0\zeta_1+\zeta_2 +A_1^{-1}\left( ((\mathcal{L}^1)^*)(\nabla \phi)+(\mathcal{L}^{2,1})^*\phi+(\mathcal{L}^{3})^*\phi-(\mathcal{L}^{3})^*(\mathcal{T}\zeta_2)\right) =\overline{g}& \text{in}\ \omega,\\
\lambda_0\zeta_2-A_1\zeta_1+A_2\zeta_2+\mathcal{T}^*(\mathbb{T}(\phi,r)n) +(\mathcal{L}^{2,2})^*\phi=\overline{h}& \text{in}\ \omega,\\
(\phi-\mathcal{T}\zeta_2)_{n}=0& \text{on} \ \partial\Omega,\\
\left[ 2\nu D(\phi) n + \beta (\phi-\mathcal{T}\zeta_2)\right]_{\tau}=0 &\text{on} \ \partial\Omega.
\end{array}
\right.
\end{equation*}
We precise that $(\phi,\zeta_1,\zeta_2)$ is solution of the system
\begin{equation}
(\lambda_0-A_S^*)(\phi,\zeta_1,\zeta_2)=\mathbb{P}(\overline{f},\overline{g},\overline{h}).
\end{equation}
If $\beta_1>0$, the operator $B^*$ is given by 
$$
B^*(\phi,\zeta_1,\zeta_2)=\left. m\left(  (\mathbb{T}(\phi,r)n)\cdot n-\int_{\Gamma_0}m\mathbb{T}(\phi,r)n\cdot n \ d\Gamma\right)n +m(\mathbb{T}(\phi,r)n)_{\tau}\right| _{\Gamma_0},
$$
and if $\beta_1=0$, we get
$$
B^*(\phi,\zeta_1,\zeta_2)=\left. m\left(  (\mathbb{T}(\phi,r)n)\cdot n-\int_{\Gamma_0}m\mathbb{T}(\phi,r)n\cdot n \ d\Gamma\right)n \right| _{\Gamma_0},
$$
where $r\in H^1(\Omega)/\mathbb{R}$, is the solution of the problem
\begin{equation}
\begin{pmatrix}
\nabla r\\ 0 \\ -\mathcal{T^*}(rn)
\end{pmatrix}=(I-\mathbb{P})\mathcal{A}^*_S\begin{pmatrix}
\phi\\ \zeta_1 \\ \zeta_2
\end{pmatrix}.
\end{equation}
Now, we establish a first result of stabilization of the linear system \eqref{oseen2}, \eqref{oseenb2} and \eqref{oseenc}. To do so, it suffices to verify that $(A_S,B)$ satisfies the Fattorini-hautus criterion \eqref{UCstab*} and to apply \cite[Theorem 1.1 ]{hal-02545562}.

Let $\sigma>0$ and $ \lambda\in\mathbb{C}$ such that
$\Re\lambda\geq -\sigma $. Let $(\phi,r,\zeta_1,\zeta_2)$ solution of the system
\begin{equation*}
\left\{
\begin{array}{rl}
\lambda \phi-\nabla\cdot(\mathbb{T}(\phi,r))-(w^S\cdot \nabla)\phi+(\nabla w^S)^*\phi=0& \text{in}\ \Omega,\\
\nabla \cdot \phi =0& \text{in}\ \Omega,\\
\lambda\zeta_1+\zeta_2 +A_1^{-1}\left( ((\mathcal{L}^1)^*)(\nabla \phi)+(\mathcal{L}^{2,1})^*\phi-(\mathcal{L}^{3})^*\phi+(\mathcal{L}^{3})^*(\mathcal{T}\zeta_2)\right)=0& \text{in}\ \omega,\\
\lambda\zeta_2-A_1\zeta_1+A_2\zeta_2+\mathcal{T}^*(\mathbb{T}(\phi,r)n) +(\mathcal{L}^{2,2})^*\phi=0& \text{in}\ \omega,\\
(\phi-\mathcal{T}\zeta_2)_{n}=0& \text{on} \ \partial\Omega,\\
\left[ 2\nu D(\phi) n + \beta (\phi-\mathcal{T}\zeta_2)\right]_{\tau}=0 &\text{on} \ \partial\Omega.
\end{array}
\right.
\end{equation*}
and assume that $B^*(\phi,\zeta_1,\zeta_2)=0$, it implies that for $\beta_1\geq 0$, we have
\begin{equation}
\label{autus}
m\left(  (\mathbb{T}(\phi,r)n)\cdot n-\int_{\Gamma_0}m\mathbb{T}(\phi,r)n\cdot n \ d\Gamma\right)n +m(\mathbb{T}(\phi,r)n)_{\tau}=0.
\end{equation}
We set
$$
c(\phi,r)=\int_{\Gamma_0}m\mathbb{T}(\phi,r)n\cdot n \ d\Gamma.
$$
The equation \eqref{autus} becomes
\begin{equation}
\label{autus1}
m (\mathbb{T}(\phi,r-c(\phi,r))n)_{n} +m(\mathbb{T}(\phi,r-c(\phi,r))n)_{\tau}=0.
\end{equation}
We deduce that 
$$
\mathbb{T}(\phi,r-c(\phi,r))n=0,\ \text{ on } \Gamma_0.
$$
Then, from \cite{FabreLebeau}, we get
$$
\phi=0,\quad \text{in } \Omega,\quad r= c(\phi,r),\quad \text{on } \Omega.
$$
In particular
\begin{equation}
\mathbb{T}(\phi,r-c(\phi,r))n=0,\ \text{on } \partial\Omega.
\end{equation}
Thus, we obtain
\begin{equation}
\label{msk}
\mathcal{T}^*(\mathbb{T}(\phi,r)n)=-c(\phi,r)\mathcal{T}^*(n),\ \text{on } \partial\Omega.
\end{equation}
From the definition of the operator $\mathcal{T}^*$ given in \eqref{t*}, we have that $\mathcal{T}^*(n)=0$. Then, from \eqref{msk}, we obtain
 \begin{equation}
\label{msk1}
\mathcal{T}^*(\mathbb{T}(\phi,r)n)=0,\ \text{on } \partial\Omega.
\end{equation}
Moreover, we have $\mathcal{T}\zeta_2=0$ on $\partial\Omega$, then $\zeta_2=0$. The equation \eqref{msk1} implies that $A_1\zeta_1=0$ and taking into account the periodicity of $\zeta_1$, we find $\zeta_1=0$. Then, the condition \eqref{UCstab*} is verified. We deduce then the following theorem.
\begin{Theorem}\label{thmain1}
Let $t_0>0$, $\sigma>0$ and $W^0\in \mathbb{V}$. Then, there exists $N_+\in\mathbb{N}^*$, $K\in L^\infty_{\rm loc}(\mathbb{R}^2;\mathcal{L}(\mathbb{H}))$, $(\phi_k,\zeta^k_1,\zeta_2^k)\in \mathcal{D}(A_S^*)$ and $v_k\in B^*\left(\mathcal{D}(A_S^*)\right)$, $k=1,\ldots,N_+$, 
such that
\begin{equation}
\label{fb2}
v(t)=1_{[t_0,+\infty)}(t) \sum_{k=1}^{N_+} \left(W(t-t_0)+\int_0^{t-t_0} K(t-t_0,s)W(s) \ ds,\begin{pmatrix}
\phi_k\\\zeta^k_1\\\zeta_2^k
\end{pmatrix}\right)_{\mathbb{H}} v_k,
\end{equation}
stabilizes exponentially the system \eqref{dyn} and we have
\begin{equation}\label{ijt201}
\|W\|_{L^2_\sigma (0,\infty;\mathcal{D}(A_S))\cap C^0_\sigma ([0,\infty);\mathbb{V})\cap H^1_\sigma (0,\infty;\mathbb{H})}
\leq C\left(\|W^0\|_{\mathbb{V}} + \|F\|_{L^2_{\sigma}(0,\infty;\mathbb{H})}\right).
\end{equation}
\end{Theorem}
Finally, using \eqref{ijt201}, \eqref{estim*} and \eqref{fh}, we obtain the following corollary.
\begin{Corollary}
Let $\gamma_0>0$. For all $\gamma\in [0,\gamma_0[$, the solution $(\widetilde{w},\widetilde{q},\xi)$ of the system \eqref{oseenl}, \eqref{oseenbl}, \eqref{nav5c} verifies the estimate
\begin{multline}
\label{es1}
\left\| (\widetilde{w},\widetilde{q},\xi)\right\|_{\mathcal{X}_{\infty,\gamma}}
\leq C\bigg( \left\|(w^0,\xi^0,\xi^1) \right\|_{\mathbb{V}}+\left\| \widetilde{f}\right\|_{L_{\gamma}^2(0,+\infty;[L^2(\Omega)]^3)}+\left\| \widetilde{h}\right\|_{L_{\gamma}^2(0,+\infty;L^2(\omega))}\\+\left\|\widetilde{g}\right\|_{W_{\gamma}^{1/4}(0,+\infty;[H^{1/2}(\partial\Omega)]^3,[L^2(\partial\Omega)]^3)} \bigg).
\end{multline}
\end{Corollary}
\section{Fixed point}
\label{ptfix}
Let $\mathcal{Y}_\infty$ and $\mathcal{B}_{\infty,R}$ be given respectively by
\begin{equation}\label{tak6.3*}
\mathcal{Y}_\infty= L^2_{\gamma}(0,\infty;[L^2(\Omega)]^3)\times  W_{\gamma}^{1/4}(0,\infty;[H^{1/2}(\partial\Omega)]^3,[L^2(\partial\Omega)]^3)\times L^2_{\gamma}(0,\infty;L^2(\omega)).
\end{equation} 
\begin{equation}\label{tak6.2*}
\mathcal{B}_{\infty,R}=\{ (\widetilde{f},\widetilde{g},\widetilde{h})\in\mathcal{Y}_\infty \; |\; \left\| (\widetilde{f},\widetilde{g},\widetilde{h})\right\|_{\mathcal{Y}_\infty}\leq R \}.
\end{equation}
Let $(\widetilde{f},\widetilde{g},\widetilde{h})\in 	\mathcal{B}_{\infty,R}$ and let $(u,p,\xi)$ be the solution of the system \eqref{oseenl}, \eqref{oseenbl}, \eqref{nav5c} associated to $(\widetilde{f},\widetilde{g},\widetilde{h})$. From \eqref{es1}, we obtain
\begin{equation}
\label{es}
\left\|(u,p,\xi)\right\|_{\mathcal{X}_{\infty,\gamma}}\leq CR,
\end{equation}
where we supposed that
$$
\left\|(w^0,\xi^0,\xi^1) \right\|_{\mathbb{V}}\leq R.
$$
From \eqref{es}, we have the following estimates
\begin{multline}
\label{im1infty*}
\left\|\xi\right\|_{L_\gamma^\infty(0,\infty;L^\infty(\omega))}
+
\left\|\partial_{s_j}\xi\right\|_{L_\gamma^\infty(0,\infty;L^\infty(\omega))} 
+\left\|\partial_{s_js_k}^2\xi\right\|_{L_\gamma^\infty(0,\infty;L^2(\omega))} 
\\+\left\|\partial_{s_js_ks_i}^3\xi\right\|_{L_\gamma^\infty(0,\infty;L^2(\omega))}\leq C\left\|\xi \right\|_{L_\gamma^\infty(0,\infty;H^3(\omega))} \leq C R,
\end{multline}
\begin{equation}
\label{im1infty*1}
\left\|\partial_t\xi\right\|_{L_\gamma^4(0,\infty;L^\infty(\omega))}
+\left\|\partial^2_{ts_j}\xi\right\|_{L_\gamma^6(0,\infty;L^2(\omega))}\leq C R,
\end{equation}
\begin{equation}
\label{im1infty**}
\left\|\xi\right\|_{H^{7/8}_\gamma(0,\infty;L^\infty(\omega))}
+
\left\|\partial_{s_j}\xi\right\|_{H^{7/8}_\gamma(0,\infty;L^\infty(\omega))} 
+\left\|\partial_{s_js_k}^2\xi\right\|_{H^{7/8}_\gamma(0,\infty;L^{8/3}(\omega))} 
\leq C R,
\end{equation}
\begin{equation}
	\label{im1infty***}
	\left\|\xi\right\|_{L^\infty_\gamma(0,\infty;H^2(\omega))}
	+
	\left\|\partial_{s_j}\xi\right\|_{L^\infty_\gamma(0,\infty;H^{3/2}(\omega))} 
	+\left\|\partial_{s_js_k}^2\xi\right\|_{L^\infty_\gamma(0,\infty;H^{1/2}(\omega))} 
	\leq C R.
\end{equation}
In particular, taking into account the condition \eqref{smallnes}, there exists $R_0>0$ such that, if $R\leq R_0$, then
\begin{equation}\label{denominfty*}
\left\| \frac{1}{1+\eta}\right\|_{L^\infty(0,\infty;L^\infty(\omega))}\leq C(R_0).
\end{equation}
In fact, if $1+\eta^S>\overline{\delta}$, then $R_0$ is chosen such that $0<R_0<\overline{\delta}$.

Here, we recall a well known result about product in Sobolev spaces.
\begin{Proposition}
\label{prod}
Let $s\geq0$, $s_1\geq s$ and $s_2\geq s$ such that $s_1+s_2>s+1/2$. Let $\mathfrak{X}_1$, $\mathfrak{X}_2$ and $\mathfrak{X}_3$ be three Banach spaces such that for all $f\in \mathfrak{X}_2 $ and $g\in \mathfrak{X}_3 $, we have
$$
\left\|fg\right\|_{ \mathfrak{X}_1} \leq C \left\|f\right\|_{ \mathfrak{X}_2} \left\|g\right\|_{ \mathfrak{X}_3}.
$$
Then
$$
\forall u\in H^{s_1}(0,\infty;\mathfrak{X}_2),\ \forall v\in H^{s_2}(0,\infty;\mathfrak{X}_3),\quad \left\| uv\right\|_{H^{s}(0,\infty;\mathfrak{X}_1)}\leq C \left\| u\right\|_{H^{s_1}(0,\infty;\mathfrak{X}_2)}\left\| v\right\|_{H^{s_2}(0,\infty;\mathfrak{X}_3)}.
$$
\end{Proposition}
\begin{proof}
The proof is similar to the one given in \cite[Theorem 2]{MR500121} extending data by reflexion on $\mathbb{R}$.
\end{proof}
\begin{Lemma}
\label{fff1-} 
Let $X$ and $Y$ given in \eqref{chg2}. 
Then, we have
\begin{equation}
\label{ep1-}
\nabla Y(X)-I_3=\ell^{(1)}(\xi,\partial_s\xi)+\epsilon^{(1)}(\xi,\partial_s\xi),
\end{equation}
\begin{equation}
\label{ep21-}
\frac{\partial^2 Y_3}{\partial x_k\partial x_j}(X)=\ell^{(2)}(\xi, \partial_s\xi,\partial_{ss}^2\xi)+\epsilon^{(2)}(\xi,\partial_s\xi,\partial_{ss}^2\xi),
\end{equation}
\begin{equation}
\label{ep3-}
\frac{\partial^2 a_{ik}}{\partial x_m\partial x_l}(X)=\ell^{(3)}(\xi,\partial_s\xi,\partial^2_{ss}\xi,\partial^3_{sss}\xi)+\epsilon^{(3)}(\xi,\partial_s\xi,\partial^2_{ss}\xi,\partial^3_{sss}\xi),
\end{equation}
\begin{equation}
\label{ep4-}
\partial_tY=\ell^{(13)}(\partial_t\xi)+\epsilon^{(13)}(\xi,\partial_t\xi),
\end{equation}
\begin{equation}
\label{ep5-}
\partial_t a_{ik}(X)=\ell^{(14)}(\partial_t\xi,\partial^2_{ts}\xi)+\epsilon^{(14)}(\xi,\partial_t\xi,\partial^2_{ts}\xi),
\end{equation}
where the operators $l^{(i)}$ have the form \eqref{17:17} and $\epsilon^{(i)}$ are given formally by
$$
\epsilon^{(1)}(\xi,\partial_s\xi)= O(\xi^{m_1}(\partial_s\xi)^{m_2}),\quad m_1+m_2\geq 2,
$$
$$
\epsilon^{(2)}(\xi,\partial_s\xi,\partial^2_{ss}\xi)= O(\xi^m_1(\partial_s\xi)^{m_2}(\partial^2_{ss}\xi)^{m_3}),\quad m_1+m_2+m_3\geq 2,
$$
$$
\epsilon^{(3)}(\xi,\partial_s\xi,\partial^2_{ss}\xi,\partial^3_{sss}\xi)= O(\xi^m_1(\partial_s\xi)^{m_2}(\partial^2_{ss}\xi)^{m_3}(\partial^3_{sss}\xi)^{m_4}),\quad m_1+m_2+m_3+m_4\geq 2,
$$
$$
\epsilon^{(13)}(\xi,\partial_t\xi)= O(\xi^{m_1}(\partial_t\xi)^{m_2}),\quad m_1+m_2\geq 2,
$$
$$
\epsilon^{(14)}(\xi,\partial_t\xi,\partial^2_{ts}\xi)= O(\xi^{m_1}(\partial_t\xi)^{m_2}(\partial^2_{ts}\xi)^{m_3}),\quad m_1+m_2+m_3\geq 2.
$$
\end{Lemma}

\begin{proof}
We recall that $\xi=\eta-\eta^S$, then we have
\begin{equation}
\label{17:37}
\frac{1}{1+\eta}=\left(\frac{1}{1+\eta^S}-\frac{\xi}{(1+\eta^S)^2} \right)+\frac{\xi^2}{(1+\eta^S)^2(1+\eta)}. 
\end{equation}
After standard calculations, 
$\nabla Y(X)-I_3$ writes
\begin{equation}
\label{tak4.6*}
\frac{\partial Y_3}{\partial x_3 }(X)-1=\frac{\eta^S-\eta}{1+\eta}=-\frac{\xi}{1+\eta^S}+\frac{\xi^2}{(1+\eta^S)^2}-\frac{\xi^3}{(1+\eta^S)^2(1+\eta)},
\end{equation}
\begin{equation}
\label{tak4.7*}
\frac{\partial Y_3}{\partial x_j}(X)=-y_3\frac{\partial_{s_j}\xi}{1+\eta^S}+y_3\partial_{s_j}\eta^S\frac{\xi}{(1+\eta)(1+\eta^S)}+y_3\partial_{s_j}\xi\frac{\xi}{(1+\eta)(1+\eta^S)},\quad \text{ $j=1,2$}.
\end{equation}
Using \eqref{17:37}, we obtain \eqref{ep1-} where $\epsilon^{(1)}(\xi,\partial_s\xi)$ contains the terms of the form
\begin{equation}
\label{sof1}
\gamma^1(\eta^S)\frac{\xi^{m_1}}{(1+\xi+\eta^S)^{\alpha_1}},\quad \gamma^2(\eta^S)\frac{\xi ^{m_2}\partial_s\xi}{(1+\xi+\eta^S)^{\alpha_2}},\quad m_1\geq 2,\ m_2 \geq 1,\ \alpha_i\geq 1,\ i=1,2,
\end{equation}
where $\gamma^i$, $i=1,2$ are functions of regularity $C^3(\omega)$.
We have for all $k,j\in \{1,2\}$,
\begin{multline}
\frac{\partial^2 Y_3}{\partial x_k\partial x_j}(X)
= -y_3\frac{\partial^2_{s_j s_k}\xi}{(1+\eta^S)}
+y_3\partial_{s_k}\eta\frac{\partial_{s_j}\xi}{(1+\eta)(1+\eta^S)}
+ y_3\partial_{s_j}\eta\frac{\partial_{s_k}\xi} {(1+\eta)(1+\eta^S)}
\\
+y_3\xi\left( \frac{\partial^2_{s_k s_j}\eta}{(1+\eta^S)(1+\eta)}-2\frac{\partial_{s_k}\eta\partial_{s_j}\eta}{(1+\eta^S)(1+\eta)^2}\right).
\end{multline}
Then, we deduce \eqref{ep21-} where $\epsilon^{(2)}(\xi,\partial_s\xi,\partial^2_{ss}\xi)$ had the terms
\begin{equation}
\label{sof2}
\gamma^3(\eta^S)\frac{\xi^{m_1}(\partial_s\xi)^{m_2}}{(1+\xi+\eta^S)^{\alpha_1}},\ m_1+m_2\geq 2,\quad \gamma^4(\eta^S)\frac{\xi^{m_3}\partial^2_{ss}\xi}{(1+\xi+\eta^S)^{\alpha_2}}, \ m_3\geq 1, \ \alpha_i\geq 1, \ i=1,2,
\end{equation}
where $\gamma^i$, $i=3,4$ are $C^2(\omega)$ functions.	
The third derivative $	\frac{\partial^3 Y_3}{\partial x_i\partial x_k\partial x_j}(X)$ admits the terms:
\begin{multline}
y_3\frac{\partial^3_{s_j s_ks_i}\xi}{(1+\eta^S)},\ y_3\frac{\partial_{s_i}\eta\partial^2_{s_js_k}\xi}{(1+\eta)(1+\eta^S)},\ y_3\frac{\partial^2_{s_is_k}\eta\partial_{s_j}\xi}{(1+\eta)(1+\eta^S)}, \ y_3\frac{\partial_{s_k}\eta\partial_{s_i}\eta\partial_{s_j}\xi}{(1+\eta)^2(1+\eta^S)},\\ y_3\frac{\partial^2_{s_js_k}\eta\partial_{s_i}\xi}{(1+\eta)^2(1+\eta^S)},\ y_3\frac{\xi\partial^3_{s_j s_ks_i}\eta }{(1+\eta^S)(1+\eta)}, \
y_3\frac{\xi\partial^2_{s_ks_j}\eta\partial_{s_i}\eta}{(1+\eta)^2(1+\eta^S)},\ y_3\frac{\xi\partial_{s_k}\eta\partial_{s_j}\eta\partial_{s_i}\eta}{(1+\eta)^4(1+\eta^S)}.
\end{multline}
Then, we get \eqref{ep3-}  where $\epsilon^{(3)}(\xi,\partial_s\xi,\partial^2_{ss}\xi,\partial^3_{sss}\xi)$ contains the terms
\begin{multline}
\label{sof3}
\gamma^5(\eta^S)\frac{\xi^{m_1}(\partial_s\xi)^{m_2}}{(1+\xi+\eta^S)^{\alpha_1}},\ \gamma^6(\eta^S)\frac{\xi^{n_1}(\partial_s\xi)^{n_2}\partial^2_{ss}\xi}{(1+\xi+\eta^S)^{\alpha_2}},\ \gamma^7(\eta^S)\frac{\xi^{m_3}\partial^3_{sss}\xi}{(1+\xi+\eta^S)^{\alpha_3}},\\ m_1+m_2\geq 2,\ n_1+n_2\geq 2, \ m_3\geq 1, \ \alpha_i\geq 1, \ i=1...3,
\end{multline}
with $\gamma^i$, $i=5,...,7$ are bounded functions on $\omega$.
	
We also have, 
$$
\partial_tY_3=-y_3\frac{\partial_t\xi}{1+\eta}.
$$
We deduce \eqref{ep4-} where $\epsilon^{(13)}(\xi,\partial_t\xi)$ contains the terms
\begin{equation}
\label{sof4}
\frac{\xi^{m_1}\partial_t\xi}{(1+\xi+\eta^S)^{\alpha_1}},\quad m_1\geq 1, \ \alpha_1\geq 1. 
\end{equation}	
The terms appearing in $\partial_t a_{ik}(X)$ have the form
$$
y_3\frac{\partial_t\xi\partial_{s_j}\eta}{(1+\eta)^2}, \quad
y_3\frac{\partial_t\xi\partial_{s_j}\eta^S}{(1+\eta)(1+\eta^S)},\quad
y_3\frac{\partial^2_{ts_j}\xi}{(1+\eta)},\quad 
-\frac{(1+\eta^S)\partial_t\xi}{(1+\eta)^2}. 
$$
Then, we get \eqref{ep5-} with $\epsilon^{(14)}(\xi,\partial_t\xi,\partial^2_{ts}\xi)$ contains the terms
\begin{multline}
\label{sof5}
\gamma^8(\eta^S)\frac{\xi^{m_1} \partial^2_{ts}\xi}{(1+\xi+\eta^S)^{\alpha_1}},\ \gamma^9(\eta^S)\frac{\xi^{m_2}\partial_t\xi }{(1+\xi+\eta^S)^{\alpha_2}},\ \gamma^{10}(\eta^S)\frac{\xi^{m_3}\partial_t\xi\partial_s\xi}{(1+\xi+\eta^S)^{\alpha_3}},\\ m_1\geq1, \  m_2\geq1, \  m_3\geq0,\ \alpha_i\geq 1, \ i=1,...,3,
\end{multline}
where $\gamma^i$, $i=8,...,10$ are $C^3(\omega)$ functions.
\end{proof}
Let suppose that $R<1$.
\begin{Lemma}
\label{fff1} 
Let $\epsilon^{(i)}$ given in \cref{fff1-} and let
\begin{equation}
\label{re}
(w^S,p^S,\eta^S)\in [W^{2,\infty}(\Omega)]^3\times W^{1,\infty}(\Omega)\times C^4(\omega).
\end{equation} 
Then, we have
\begin{multline}
\label{ep11}
\left\| \epsilon^{(1)}(\xi,\partial_s\xi)\right\|_{L^\infty_\gamma(0,\infty;[L^\infty(\Omega)]^9)}+
\left\| \epsilon^{(2)}(\xi,\partial_s\xi,\partial_{ss}^2\xi)\right\|_{L^\infty_\gamma(0,\infty;L^2(\Omega))}
\\+
\left\| \epsilon^{(3)}(\xi,\partial_s\xi,\partial^2_{ss}\xi,\partial^3_{sss}\xi)\right\|_{L^\infty_\gamma(0,\infty;L^2(\Omega))}+
\left\|\epsilon^{(13)}(\xi,\partial_t\xi)\right\|_{L^4_\gamma(0,\infty;[L^\infty(\Omega)]^3)}\\+
\left\| \epsilon^{(14)}(\xi,\partial_t\xi,\partial^2_{ts}\xi)\right\|_{L^6_\gamma(0,\infty;L^2(\Omega))}\leq C R^2,
\end{multline}
\begin{multline}
\label{ep12}
\left\| \epsilon^{(1)}(\xi,\partial_s\xi)\right\|_{L^\infty_\gamma(0,\infty;[H^{3/2}(\partial\Omega)]^9)}+
\left\| \epsilon^{(2)}(\xi,\partial_s\xi,\partial_{ss}^2\xi)\right\|_{L^\infty_\gamma(0,\infty;H^{1/2}(\partial\Omega))}
\\+\left\| \epsilon^{(1)}(\xi,\partial_s\xi)\right\|_{H^{7/8}_\gamma(0,\infty;[L^\infty(\partial\Omega)]^9)}+
\left\| \epsilon^{(2)}(\xi,\partial_s\xi,\partial_{ss}^2\xi)\right\|_{H^{7/8}_\gamma(0,\infty;L^{8/3}(\partial\Omega))}\leq C R^2.
\end{multline}
\end{Lemma}
\begin{proof}
Using \eqref{re}, \eqref{im1infty*},\eqref{denominfty*} and \eqref{sof1}, we get
\begin{equation}
\label{tom1}
\left\| \epsilon^{(1)}(\xi,\partial_s\xi)\right\|_{L^\infty_\gamma(0,\infty;[L^\infty(\Omega)]^9)} \leq CR^2.
\end{equation}
From \eqref{denominfty*}, \eqref{re}, \eqref{im1infty*} and \eqref{sof2}, we find
\begin{equation}
\label{tom2}
\left\| \epsilon^{(2)}(\xi,\partial_s\xi,\partial_{ss}^2\xi)\right\|_{L^\infty_\gamma(0,\infty;L^2(\Omega))} \leq CR^2.
\end{equation}
From \eqref{denominfty*}, \eqref{re}, \eqref{im1infty*} and \eqref{sof3}, we obtain
\begin{equation}
\label{tom3}
\left\| \epsilon^{(3)}(\xi,\partial_s\xi,\partial^2_{ss}\xi,\partial^3_{sss}\xi)\right\|_{L^\infty_\gamma(0,\infty;L^2(\Omega))} \leq CR^2.
\end{equation}
From \eqref{denominfty*}, \eqref{re}, \eqref{im1infty*1}, \eqref{im1infty*} and \eqref{sof4}, we get
\begin{equation}
\label{tom4}
\left\|\epsilon^{(13)}(\xi,\partial_t\xi)\right\|_{L^4_\gamma(0,\infty;[L^\infty(\Omega)]^3)}\leq CR^2.
\end{equation}	
Using \eqref{denominfty*}, \eqref{re}, \eqref{im1infty*}, \eqref{im1infty*1} and \eqref{sof5}, we get
\begin{equation}
\label{tom5}
\left\| \epsilon^{(14)}(\xi,\partial_t\xi,\partial^2_{ts}\xi)\right\|_{L^6_\gamma(0,\infty;L^2(\Omega))}\leq C R^2.
\end{equation}
Using \eqref{tom1}, \eqref{tom2}, \eqref{tom3}, \eqref{tom4} and \eqref{tom5}, we deduce \eqref{ep11}.	
In the other hand, we have 
\begin{equation}
\left\|\xi\partial_s\xi \right\|_{L^\infty_{\gamma}(0,\infty;H^{3/2}(\omega))}\leq C\left\|\xi \right\|_{L^\infty_{\gamma}(0,\infty;H^2(\omega))}\left\|\partial_s\xi \right\|_{L^\infty_{\gamma}(0,\infty;H^{3/2}(\omega))}\leq CR^2.   
\end{equation}
Thus, using \eqref{denominfty*} and \eqref{im1infty*}, we have
\begin{equation}
\label{riri}
\left\|\frac{1}{(1+\xi+\eta^S)^{\alpha_i}} \right\|_{L^\infty_{\gamma}(0,\infty;H^{2}(\omega))}\leq C,\quad \alpha_i\geq 1.
\end{equation}
Then,
\begin{equation}
\left\|\frac{\xi^{m_1}\partial_s\xi}{(1+\xi+\eta^S)^{\alpha_i}} \right\|_{L^\infty_{\gamma}(0,\infty;H^{3/2}(\omega))}\leq CR^2.   
\end{equation}
Then, using the fact that $\partial_s\xi\in L_\gamma^\infty(0,\infty;H^2(\omega))$ and $\left\|\partial_s\xi\right\|_{L_\gamma^\infty(0,\infty;H^2(\omega))}\leq CR$, we get
\begin{equation}
\label{art1}
\left\|\xi^{m_1}(\partial_s\xi)^{m_2} \right\|_{L^\infty_{\gamma}(0,\infty;H^{3/2}(\omega))}\leq CR^2.   
\end{equation}
We have also
\begin{equation}
\left\|\xi\partial^2_{ss}\xi \right\|_{L^\infty_{\gamma}(0,\infty;H^{1/2}(\omega))}\\\leq C\left\|\xi\right\|_{L^\infty_{\gamma}(0,\infty;H^{3/2}(\omega))}\left\|\partial^2_{ss}\xi\right\|_{L^\infty_\gamma(0,\infty;H^{1/2}(\omega))}\leq CR^2 ,
\end{equation}
Then, using \eqref{riri}, we obtain
\begin{equation}
\label{art2}
\left\|\frac{\xi^{m_1}\partial^2_{ss}\xi}{(1+\xi+\eta^S)^{\alpha_i}} \right\|_{L^\infty_{\gamma}(0,\infty;H^{1/2}(\omega))}\leq CR^2 ,
\end{equation}
Furthermore, using \eqref{denominfty*}, we get
\begin{multline}
\label{riri1}
\left\|\frac{\xi}{(1+\xi+\eta^S)^{\alpha_i}} \right\|_{H^{5/4}_{\gamma}(0,\infty;L^\infty(\omega))}\leq C\left\|\frac{\xi}{(1+\xi+\eta^S)^{\alpha_i}} \right\|_{H^{2}_{\gamma}(0,\infty;L^2(\omega))}^{5/8}\left\|\frac{\xi}{(1+\xi+\eta^S)^{\alpha_i}} \right\|_{L^2_{\gamma}(0,\infty;H^4(\omega))}^{3/8}\\\leq CR,\quad \alpha_i\geq 1.
\end{multline}  
hence, we obtain
\begin{multline}
\left\|\frac{\xi\partial_s\xi}{(1+\xi+\eta^S)^{\alpha_i}} \right\|_{H^{7/8}_{\gamma}(0,\infty;L^\infty(\omega))}\leq C\left\|\frac{\xi}{(1+\xi+\eta^S)^{\alpha_i}} \right\|_{H^{5/4}_{\gamma}(0,\infty;L^\infty(\omega))}\left\|\partial_s\xi \right\|_{H^{7/8}_{\gamma}(0,\infty;L^\infty(\omega))} \\ \leq CR^2.   
\end{multline}
Since $7/8>1/2$, we use the \cref{prod}, we obtain
$$
\left\|(\partial_s \xi)^{m_2} \right\|_{H^{7/8}(0,\infty;L^\infty(\omega))}\leq C\left\|(\partial_s \xi) \right\|_{H^{7/8}(0,\infty;L^\infty(\omega))}^{m_2}.
$$ 
We deduce
\begin{equation}
\label{art3}
\left\|\frac{\xi^{m_1}(\partial_s\xi)^{m_2}}{(1+\xi+\eta^S)^{\alpha_i}}\right\|_{H^{7/8}_{\gamma}(0,\infty;L^\infty(\omega))}\leq CR^2.
\end{equation}
Moreover,
\begin{equation}
\left\|\frac{\xi\partial^2_{ss}\xi}{(1+\xi+\eta^S)^{\alpha_i}} \right\|_{H^{7/8}_{\gamma}(0,\infty;L^{8/3}(\omega))}\leq C\left\|\partial^2_{ss}\xi\right\|_{H^{7/8}_{\gamma}(0,\infty;L^{8/3}(\omega))}\left\|\xi\right\|_{H^{5/4}_{\gamma}(0,\infty;L^\infty(\omega))}\leq CR^2.
\end{equation}
From \eqref{riri1}, we have
\begin{equation}
\label{art4}
\left\|\frac{\xi^{m_1}\partial^2_{ss}\xi}{(1+\xi+\eta^S)^{\alpha_i}} \right\|_{H^{7/8}_{\gamma}(0,\infty;L^{8/3}(\omega))}\leq CR^2 ,
\end{equation}
Using \eqref{art1}, \eqref{art2}, \eqref{art3} and \eqref{art4}, we obtain \eqref{ep12}.
\end{proof}
\begin{Lemma}
\label{fff}
We have
\begin{equation}
\left\| \nabla\cdot\mathcal{N}_E(\xi)\right\|_{L^2_{\gamma}(0,\infty;[L^2(\Omega)]^3)}+\left\| \mathcal{N}_F(\xi)\right\|_{L^2_{\gamma}(0,\infty;[L^2(\Omega)]^3)} \leq C R^2,
\end{equation}
\begin{equation}
\label{estimE}
\left\|\mathcal{N}_E(\xi)\right\|_{L^2_{\gamma}(0,\infty;[L^2(\Omega)]^9)}\\\leq CR^2,
\end{equation}
\begin{equation}
\label{esti}
\left\|\mathcal{N}_E(\xi)n\right\|_{W_\gamma^{1/4}(0,\infty;[H^{1/2}(\partial\Omega)]^3;[L^2(\partial\Omega)]^3)}+\left\| \mathcal{N}_G(\xi)\right\|_{W_\gamma^{1/4}(0,\infty;[H^{1/2}(\partial\Omega)]^3;[L^2(\partial\Omega)]^3)}\leq CR^2,
\end{equation}
\begin{equation}
\label{halsey}
\left\|(\mathcal{F}(u,p,\xi),G(u,\xi+\eta^S),H(u,\xi+\eta^S)) \right\|_{\mathcal{Y}_\infty}\leq CR^2.
\end{equation}
\end{Lemma}
\begin{proof}
From \eqref{E} and \cref{fff1-}, $\mathcal{N}_E(\xi)$ admits the following terms
\begin{multline}
\label{georgina}
c^1(\eta^S,w^S)\frac{\xi^{m_1}(\partial_s\xi)^{m_2}}{(1+\xi+\eta^S)^{\alpha_1}},\ c^2(\eta^S,w^S)\frac{\xi^{n_1}(\partial_s\xi)^{n_2}\partial^2_{ss}\xi}{(1+\xi+\eta^S)^{\alpha_2}},\\ m_1+m_2\geq 2,\ n_1+n_2\geq 2,\alpha_i\geq 1, \ i=1...2, 
\end{multline}
where $c^i$, $i=1,2$ are $W^{1,\infty}(\Omega)$ functions.
	
From \eqref{denominfty*}, \eqref{re} and \eqref{im1infty*}, we obtain
\begin{equation}
\left\|\mathcal{N}_E(\xi)\right\|_{L^2_{\gamma}(0,\infty;[L^2(\Omega)]^3)} \leq C R^2.
\end{equation}
From \eqref{t6m}, \eqref{E}, \eqref{zak1}, \eqref{re} and  \cref{fff1-}, the terms that appear in $\nabla\cdot\mathcal{N}_E(\xi)$ are the following
\begin{multline*}
d^1(\eta^S,w^S)\frac{\xi^{m_1}(\partial_s\xi)^{m_2}}{(1+\xi+\eta^S)^{\alpha_1}},\ d^2(\eta^S,w^S)\frac{\xi^{n_1}(\partial_s\xi)^{n_2}\partial^2_{ss}\xi}{(1+\xi+\eta^S)^{\alpha_2}},\ d^3(\eta^S,w^S)\frac{\xi^{m_3}\partial^3_{sss}\xi}{(1+\xi+\eta^S)^{\alpha_3}},\\ m_1+m_2\geq 2,\ n_1+n_2\geq 2, \ m_3\geq 1, \ \alpha_i\geq 1, \ i=1...3,
\end{multline*}
where $d^i$, $i=1,2,3$ are bounded functions in $\Omega$.
	
From \eqref{denominfty*}, \eqref{re} and \eqref{im1infty*}, we obtain
\begin{equation}
\left\| \nabla\cdot\mathcal{N}_E(\xi)\right\|_{L^2_{\gamma}(0,\infty;[L^2(\Omega)]^3)} \leq C R^2.
\end{equation}
Moreover, from \eqref{nf}, $\mathcal{N}_F(\xi)$ has the terms
\begin{multline*}
a^1(\eta^S,w^S)\frac{\xi^{m_1}(\partial_s\xi)^{m_2}}{(1+\xi+\eta^S)^{\alpha_1}},\ a^2(\eta^S,w^S)\frac{\xi^{n_1}(\partial_s\xi)^{n_2}\partial^2_{ss}\xi}{(1+\xi+\eta^S)^{\alpha_2}}, \ a^3(\eta^S,w^S)\frac{\xi^{m_3} \partial^2_{ts}\xi}{(1+\xi+\eta^S)^{\alpha_3}},\\ a^4(\eta^S,w^S)\frac{\xi^{m_4}\partial_t\xi }{(1+\xi+\eta^S)^{\alpha_4}}, a^5(\eta^S,w^S)\frac{\xi^{m_5}\partial_t\xi\partial_s\xi}{(1+\xi+\eta^S)^{\alpha_5}},\\ m_1+m_2\geq 2,\ n_1+n_2\geq 2, \ m_3\geq 1, \ m_3\geq1, \  m_4\geq1, \  m_5\geq0,\ \alpha_i\geq 1, \ i=1...5,
\end{multline*}
where $a^i$, $i=1,...,5$ are bounded functions in $\Omega$.
	
Thus, using \eqref{denominfty*}, \eqref{re}, \eqref{im1infty*} and \eqref{im1infty*1}, we obtain
\begin{equation*}
\left\| \mathcal{N}_F(\xi)\right\|_{L^2_{\gamma}(0,\infty;[L^2(\Omega)]^9)} \leq C R^2.
\end{equation*}
The functions $\mathcal{N}_{E}(\xi)$ and $\mathcal{N}_G(\xi)$ admit terms of the form \eqref{georgina}. Then, from \cref{fff1} and \eqref{re}, we have the estimate \eqref{esti}.
	
Finally, using \cref{fff1}, \eqref{es} and the same procedure described in \cite{MR3962841}, we get \eqref{halsey}.
\end{proof}
\subsection{Proof of \cref{thmain2}}
We define the application $\Phi:\mathcal{B}_{\infty,R} \longrightarrow \mathcal{Y}_\infty$ by
\begin{multline}
\label{fi}
\Phi(\widetilde{f},\widetilde{g},\widetilde{h})= (\nabla\cdot\mathcal{N}_E(\xi)+\mathcal{N}_F(\xi)+\mathcal{F}(u,p,\xi),\\-[\mathcal{N}_E(\xi)n]_{\tau}+[\mathcal{N}_G(\xi)]_{\tau}+G(u,\xi+\eta^S),H(u,\xi+\eta^S)-\mathcal{T}^*(\mathcal{N}_E(\xi)n)).
\end{multline}
In this case, we show that for $R$ small enough
$\Phi (\mathcal{B}_{\infty,R})\subset \mathcal{B}_{\infty,R}$ and
$\Phi_{|\mathcal{B}_{\infty,R}}$ is a strict contraction.  
Let $(\widetilde{f},\widetilde{g},\widetilde{h})\in\mathcal{B}_{\infty,R}$, from \cref{fff}, we obtain
\begin{equation*}
\left\|\Phi(\widetilde{f},\widetilde{g},\widetilde{h})\right\|_{\mathcal{Y}_\infty}\leq CR^2.
\end{equation*}
We get similarly
\begin{equation*}
\left\|\Phi(\widetilde{f}^{(1)},\widetilde{g}^{(1)},\widetilde{h}^{(1)})-\Phi(\widetilde{f}^{(2)},\widetilde{g}^{(2)},\widetilde{h}^{(2)}) \right\|_{\mathcal{Y}_\infty}
\leq CR\left\|(\widetilde{f}^{(1)},\widetilde{g}^{(1)},\widetilde{h}^{(1)})-(f^{(2)},\widetilde{g}^{(2)},\widetilde{h}^{(2)}) \right\|_{\mathcal{Y}_\infty}.
\end{equation*}
for $(\widetilde{f},\widetilde{g},\widetilde{h}), (\widetilde{f}^{(i)},\widetilde{g}^{(i)},\widetilde{h}^{(i)})\in \mathcal{B}_{\infty,R}$.

Thus, we deduce the principal result of this paper.


\bibliographystyle{plain}
\bibliography{biblio}

\begin{thebibliography}{10}

\bibitem{MR2567253}
M.~Badra.
\newblock Feedback stabilization of the 2-{D} and 3-{D} {N}avier-{S}tokes
  equations based on an extended system.
\newblock {\em ESAIM Control Optim. Calc. Var.}, 15(4):934--968, 2009.

\bibitem{MR2516189}
M.~Badra.
\newblock Lyapunov function and local feedback boundary stabilization of the
  {N}avier-{S}tokes equations.
\newblock {\em SIAM J. Control Optim.}, 48(3):1797--1830, 2009.

\bibitem{MR2851895}
M.~Badra.
\newblock Abstract settings for stabilization of nonlinear parabolic system
  with a {R}iccati-based strategy. {A}pplication to {N}avier-{S}tokes and
  {B}oussinesq equations with {N}eumann or {D}irichlet control.
\newblock {\em Discrete Contin. Dyn. Syst.}, 32(4):1169--1208, 2012.

\bibitem{MB1}
M.~Badra and T.~Takahashi.
\newblock Stabilization of parabolic nonlinear systems with finite dimensional
  feedback or dynamical controllers: application to the {N}avier-{S}tokes
  system.
\newblock {\em SIAM J. Control Optim.}, 49(2):420--463, 2011.

\bibitem{MR3261920}
M.~Badra and T.~Takahashi.
\newblock Feedback stabilization of a fluid-rigid body interaction system.
\newblock {\em Adv. Differential Equations}, 19(11-12):1137--1184, 2014.

\bibitem{MR3181675}
M.~Badra and T.~Takahashi.
\newblock Feedback stabilization of a simplified 1d fluid-particle system.
\newblock {\em Ann. Inst. H. Poincar\'{e} Anal. Non Lin\'{e}aire},
  31(2):369--389, 2014.

\bibitem{MT}
M.~Badra and T.~Takahashi.
\newblock Feedback boundary stabilization of 2{D} fluid-structure interaction
  systems.
\newblock {\em Discrete Contin. Dyn. Syst.}, 37(5):2315--2373, 2017.

\bibitem{BdV}
H.~Beir\~{a}o Da~Veiga.
\newblock Regularity for {S}tokes and generalized {S}tokes systems under
  nonhomogeneous slip-type boundary conditions.
\newblock {\em Adv. Differential Equations}, 9(9-10):1079--1114, 2004.

\bibitem{BDDM}
A.~Bensoussan, G.~Da~Prato, M.C. Delfour, and S.K. Mitter.
\newblock {\em Representation and control of infinite dimensional systems}.
\newblock Systems \& Control: Foundations \& Applications. Birkh\"{a}user
  Boston, Inc., Boston, MA, second edition, 2007.

\bibitem{MR3767485}
D.~Bresch-Pietri, C.~Prieur, and E.~Tr\'{e}lat.
\newblock New formulation of predictors for finite-dimensional linear control
  systems with input delay.
\newblock {\em Systems Control Lett.}, 113:9--16, 2018.

\bibitem{MR937679}
R.~Datko.
\newblock Not all feedback stabilized hyperbolic systems are robust with
  respect to small time delays in their feedbacks.
\newblock {\em SIAM J. Control Optim.}, 26(3):697--713, 1988.

\bibitem{MR818942}
R.~Datko, J.~Lagnese, and M.P. Polis.
\newblock An example on the effect of time delays in boundary feedback
  stabilization of wave equations.
\newblock {\em SIAM J. Control Optim.}, 24(1):152--156, 1986.

\bibitem{MR3962841}
I.A. Djebour and T.~Takahashi.
\newblock On the existence of strong solutions to a fluid structure interaction
  problem with {N}avier boundary conditions.
\newblock {\em J. Math. Fluid Mech.}, 21(3):Paper No. 36, 30, 2019.

\bibitem{hal-02545562}
I.A. Djebour, T.~Takahashi, and J.~Valein.
\newblock Feedback stabilization of parabolic systems with input delay.
\newblock {\em Mathematical Control \& Related Fields}, 2021.

\bibitem{FabreLebeau}
C.~Fabre and G.~Lebeau.
\newblock Prolongement unique des solutions de l'equation de {S}tokes.
\newblock {\em Comm. Partial Differential Equations}, 21(3-4):573--596, 1996.

\bibitem{doi:10.1002/fld.1650040302}
S.F. Kistler and L.E. Scriven.
\newblock Coating flow theory by finite element and asymptotic analysis of the
  {N}avier-{S}tokes system.
\newblock {\em International Journal for Numerical Methods in Fluids},
  4(3):207--229, 1984.

\bibitem{lhachemi2019boundary}
H.~Lhachemi and R.~Shorten.
\newblock Boundary input-to-state stabilization of a damped {E}uler-{B}ernoulli
  beam in the presence of a state-delay.
\newblock \url{https://arxiv.org/pdf/1912.01117.pdf}, 2019.

\bibitem{MR2123407}
A.~Liakos.
\newblock Finite-element approximation of viscoelastic fluid flow with slip
  boundary condition.
\newblock {\em Comput. Math. Appl.}, 49(2-3):281--294, 2005.

\bibitem{Lions}
J.-L. Lions and E.~Magenes.
\newblock {\em Probl\`emes aux limites non homog\`enes et applications. {V}ol.
  2}.
\newblock Travaux et Recherches Math\'{e}matiques, No. 18. Dunod, Paris, 1968.

\bibitem{navier1823}
C.L.M.H. Navier.
\newblock M\'{e}moire sur les lois du mouvement des fluides.
\newblock {\em M\'{e}moires de l'Acad\'{e}mie Royale des Sciences de l'Institut
  de France}, 6(1823):389--440, 1823.

\bibitem{PZ}
A.~Pazy.
\newblock {\em Semigroups of linear operators and applications to partial
  differential equations}, volume~44 of {\em Applied Mathematical Sciences}.
\newblock Springer-Verlag, New York, 1983.

\bibitem{MR3936420}
C.~Prieur and E.~Tr\'{e}lat.
\newblock Feedback stabilization of a 1-{D} linear reaction-diffusion equation
  with delay boundary control.
\newblock {\em IEEE Trans. Automat. Control}, 64(4):1415--1425, 2019.

\bibitem{MR2247716}
J.-P. Raymond.
\newblock Feedback boundary stabilization of the two-dimensional
  {N}avier-{S}tokes equations.
\newblock {\em SIAM J. Control Optim.}, 45(3):790--828, 2006.

\bibitem{MR2335090}
J.-P. Raymond.
\newblock Feedback boundary stabilization of the three-dimensional
  incompressible {N}avier-{S}tokes equations.
\newblock {\em J. Math. Pures Appl. (9)}, 87(6):627--669, 2007.

\bibitem{JPR}
J.-P. Raymond.
\newblock Feedback stabilization of a fluid-structure model.
\newblock {\em SIAM J. Control Optim.}, 48(8):5398--5443, 2010.

\bibitem{RaymondStokes}
J.-P. Raymond.
\newblock Stokes and {N}avier-{S}tokes equations with a nonhomogeneous
  divergence condition.
\newblock {\em Discrete Contin. Dyn. Syst. Ser. B}, 14(4):1537--1564, 2010.

\bibitem{MR2384339}
Y.~Shibata and S.~Shimizu.
\newblock On the {$L_p$}-{$L_q$} maximal regularity of the {N}eumann problem
  for the {S}tokes equations in a bounded domain.
\newblock {\em J. Reine Angew. Math.}, 615:157--209, 2008.

\bibitem{RS}
R.~Shimada.
\newblock On the {$L_p$}-{$L_q$} maximal regularity for {S}tokes equations with
  {R}obin boundary condition in a bounded domain.
\newblock {\em Math. Methods Appl. Sci.}, 30(3):257--289, 2007.

\bibitem{T}
R.~Temam.
\newblock {\em Navier-{S}tokes equations}, volume~2 of {\em Studies in
  Mathematics and its Applications}.
\newblock North-Holland Publishing Co., Amsterdam-New York, revised edition,
  1979.
\newblock Theory and numerical analysis, With an appendix by F. Thomasset.

\bibitem{MR884296}
R.~Verf\"{u}rth.
\newblock Finite element approximation of incompressible {N}avier-{S}tokes
  equations with slip boundary condition.
\newblock {\em Numer. Math.}, 50(6):697--721, 1987.

\bibitem{MR500121}
J.-L. Zolesio.
\newblock Multiplication dans les espaces de {B}esov.
\newblock {\em Proc. Roy. Soc. Edinburgh Sect. A}, 78(1-2):113--117, 1977/78.

\end{thebibliography}
	
\end{document}